\documentclass[11pt]{article}

\usepackage{amsmath,amsthm,amsfonts,amssymb,color}

\parskip=\medskipamount
\def\del{{\nabla}}

\def\fpd#1#2{\frac{\partial #1}{\partial #2}}
\def\R{{\mathbb R}}

\def\hook{{\mathchoice{\vrule height 0pt depth 0.4pt width 3pt
\vrule height 5pt depth 0.4pt \kern 3pt} {\vrule height 0pt depth
0.4pt width 3pt \vrule height 5pt depth 0.4pt \kern 3pt} {\vrule
height 0pt depth 0.2pt width 1.5pt \vrule height 3pt depth 0.2pt
width 0.2pt \kern 1pt} {\vrule height 0pt depth 0.2pt width 1.5pt
\vrule height 3pt depth 0.2pt width 0.2pt \kern 1pt} }}

\def\H#1{{#1}^{\scriptscriptstyle H}}
\def\V#1{{#1}^{\scriptscriptstyle V}}

\theoremstyle{plain}
\newtheorem{thm}{Theorem}[section]
\newtheorem{lem}[thm]{Lemma}
\newtheorem{propn}[thm]{Proposition}
\newtheorem{cor}[thm]{Corollary}

\theoremstyle{definition}
\newtheorem{defn}[thm]{Definition}
\newtheorem{xmpl}{Example}
\newcommand{\pd}[2]{\frac{\partial#1}{\partial#2}}
\newcommand{\hnabla}{\hat{\nabla}}

\begin{document}

\title{New progress in the inverse problem
in the calculus of variations}

\author{Thoan Do and Geoff Prince
\thanks{Email: {\tt t4do\char64 students.latrobe.edu.au,
geoff\char64 amsi.org.au}}
\\Department of Mathematics and Statistics, La Trobe University,\\ Victoria 3086,
Australia.\\ * The Australian Mathematical Sciences Institute,\\ c/o The University of Melbourne,\\Victoria 3010, Australia.}

\maketitle

\begin{quote}
{\bf Abstract.} {\small We present a  new class of solutions for the inverse problem in the calculus of variations in arbitrary dimension $n$. This is the problem of determining the existence and uniqueness of Lagrangians for systems of $n$ second order ordinary differential equations. We also provide a number of new theorems concerning the inverse problem using exterior differential systems theory (EDS). Concentrating on the differential step of the EDS process, our new results provide a significant advance in the understanding of the inverse problem in arbitrary dimension. In particular, we indicate how to generalise Jesse Douglas's famous solution for $n=2$.
We give some non-trivial examples in dimensions 2,3 and 4. We finish with a new classification scheme for the inverse problem in arbitrary dimension.}

\end{quote}

\section{Introduction: the inverse problem}

The inverse problem in the calculus of variations can be expressed as follows. Given a system of second-order ordinary differential equations
\[
\ddot x^a = F^a (t, x^b, \dot x^b), \ \ a, b = 1, \dots, n,
\]
the question is whether the solutions of this system are also the solutions of the Euler Lagrange equations,
$$
\frac{d}{dt}\Bigl(\pd{L}{\dot{x}^a}\Bigr) - \pd{L}{x^a} = 0,
$$
for some regular Lagrangian $L(t,x^b,\dot{x}^b)$. This problem was first proposed by Helmholtz \cite{HH01} in 1887. He considered whether the equations {\em in the form presented} were Euler-Lagrange. In the case of single equations Helmholtz found necessary conditions for this to be true. It is not well-known that Sonin \cite{Son} solved this one-dimensional problem the previous year in a more general form, although Hirsch \cite{Hirsch02} is credited with the so-called multiplier version of the inverse problem, which is the focus of this paper. He addressed the uniqueness and existence of a non-degenerate multiplier matrix, $g_{ab}(t, x^c, \dot x^c)$, satisfying
\[
g_{ab} (\ddot x^b - F^b) \equiv \frac{d}{dt} \left(\frac{\partial
L}{
\partial \dot x^a}\right) - \frac{\partial L}{\partial x^a}.
\]
Necessary and
sufficient conditions for the existence of a regular Lagrangian, according to Douglas \cite{D41} and to Sarlet \cite{S82},  are that this multiplier satisfy

\begin{equation}
g_{ab} = g_{ba}, \quad \Gamma(g_{ab}) =
g_{ac}\Gamma_b^c+g_{bc}\Gamma_a^c, \quad g_{ac}\Phi_b^c  =
g_{bc}\Phi_a^c, \quad \frac{\partial g_{ab}}{\partial\dot x^c}  =
\frac{\partial g_{ac}}{\partial\dot x^b}, \label{Helmholtz1}
\end{equation}
where
\[
\Gamma_b^a := -{\frac{1}{2}}\frac{\partial F^a}{\partial \dot
x^b}, \quad \Phi_b^a := -\frac{\partial F^a}{\partial x^b} -
\Gamma_b^c\Gamma_c^a - \Gamma(\Gamma_b^a),
\]
and where
\[
\Gamma := \frac{\partial}{\partial t} + \dot{x}^a
\frac{\partial}{\partial x^a} + F^a \frac{\partial}{\partial
\dot{x}^a}.
\]
These conditions, along with non-degeneracy, are commonly referred to as the {\it Helmholtz conditions}. If one or more matrices $g_{ab}$ are found that satisfy these four Helmholtz conditions, then there exists one (or more) Lagrangians related to these
matrices by the expression,
$$
\frac{\partial^2 L}{\partial \dot{x}^a \partial \dot{x}^b} = g_{ab}.
$$
Integrating this for a given $g_{ab}$ we see that the related Lagrangian $L$ is only defined up to the addition of a total time derivative of an arbitrary function of $t, x^a$.

The multiplier problem was completely solved by Douglas in 1941 for the two dimensional case (see \cite{D41}), that is, a pair of second order equations on the plane. He divided the problem into four primary cases (I to IV) according to the properties of the matrix $\Phi^a_b$. The corresponding solution for higher dimensions, even for dimension 3, remained unsolved until the late nineteen nineties when some arbitrary $n$ subcases were elaborated~\cite{CPST, SaCraMa,AT92}.

The current attacks on this problem, dating back to the 1980s, involve the creation and use of various differential geometric tools. We offer the reader the following references which give some perspective on these developments \cite{AT92,Cramp84,CSMBP94,GM2,KP08,MFLMR90,P00,STP}.

The current situation, in the framework of \cite{CSMBP94}, is that the following dimension $n$ situations are solved in the sense of Douglas.
\begin{itemize}
\item[1.] $\Phi$ is a multiple of the identity. The multiplier is determined by $n$ arbitrary functions each of $n+1$  variables. This is the extension of Douglas's case I. See \cite{Al06,AT92,SaCraMa}.
\item[2.] $\Phi$ is diagonalisable with distinct eigenvalues and ``integrable eigenspaces". The multiplier is determined by $n$ arbitrary functions each of 2 variables. This is the extension of Douglas's case IIa1. See \cite{CPST,Al03}.
\item[3.] There are many non-existence results depending on technical conditions on $\Phi$. See \cite{PK07}.
\end{itemize}

In the context of our current paper Anderson and Thompson~\cite{AT92} made a significant breakthrough by applying the method of exterior differential systems (EDS) to the inverse problem. They illustrated the effectiveness of the method by many concrete examples. In their paper, however only Douglas's case I where $\Phi$ is a multiple of the identity was generalised to arbitrary $n$. Aldridge \cite{Al03} and Aldridge et al \cite{Al06} pursued this EDS approach further using the connection of Massa and Pagani~\cite{MaPa94}. While the thesis \cite{Al03} re-investigated Douglas' case IIa2 in dimension 2 and recovered case IIa1 for arbitrary  dimension, the paper \cite{Al06} focused only on the first step of EDS process, the so-called differential ideal step.

In this paper we give results for the cases where some eigen co-distributions are integrable and some are not by using exterior differential system (EDS) theory. We provide a number of quite general theorems and propositions concerning the inverse problem via EDS in Section \ref{section-IP-EDS}. Among these interesting results, Theorem \ref{DI-first-step-cond} is particularly useful in indicating non-existence cases. In addition, it suggests a scheme to replace the four cases of Douglas in the treatment of the arbitrary dimension problem. Up until now it has not been obvious how to do this because of the low dimensionality, however we find that when $\Phi$ is diagonalisable the integrability of the co-distributions must be considered first and the termination of the differential step second. Our scheme is detailed in Section \ref{classification-scheme}.  In Section \ref{solve-case-2a2-n} we particularly solve the inverse problem for the system of second-order ODE in arbitrary dimension of which $\Phi$ is diagonalisable with distinct eigenvalues and exactly one co-distribution is non-integrable. This can be seen as an extension of Douglas's case IIa2 in which, for $n=2$, one co-distribution is integrable and one is not. The results are given in Theorem \ref{case-2a2-n-results}.

Of course, in general, $\Phi$ will not be diagonalisable over the reals. But when it is, the eigenvalues will generally be distinct and none of the eigen co-distributions will be integrable. Even in this situation there are cases where solutions exist. Indeed, it is remarkable that so many of the cases are well populated with variational equations. We will illustrate our own results with a number of examples in Section \ref{examples}.

\section{Geometric formulation}

We now briefly describe the basics of our geometrical calculus, for more details we refer to the book chapter \cite{KP08}. We are analysing a system of second-order ordinary differential equations
\begin{align}\label{SODE1}
\ddot x^a = F^a (t, x^b, \dot x^b), \ \ a, b = 1, \dots, n,
\end{align}
for some smooth $F^a$ on a manifold $M$ with generic local coordinates $(x^a)$. The evolution space $E:=\R \times TM$ has adapted coordinates $(t,x^a,\dot x^a)$ or $(t,x^a,u^a)$. We construct on $E$ from the system of equations \eqref{SODE1} a second-order differential equation field (SODE):
\begin{equation}\label{SODE2}
\Gamma := \frac{\partial}{\partial t} + u^a
\frac{\partial}{\partial x^a} + F^a \frac{\partial}{\partial
u^a}.
\end{equation}
The SODE produces on $E$ a nonlinear connection with components $\Gamma_b^a := -{\frac{1}{2}}\frac{\partial F^a}{\partial u^b}$. The evolution space $E$ has a number of natural structures. The contact and vertical structures are combined in the {\em vertical endomorphism} $S$, a (1,1) tensor field on $E$, with coordinate expression $S:=V_a\otimes\theta^a$, where $V_a:=\frac{\partial}{\partial u^a}$ are the vertical basis fields and $\theta^a:=dx^a-u^adt$ are local contact forms. It is shown in \cite{Cramp84} the first order deformation $\mathcal{ L}_\Gamma S$ has eigenvalues $0, 1$ and $-1$. The eigenspaces corresponding to eigenvalue $0, 1$ and $-1$ are $Sp\{\Gamma\}$, the {\em vertical subspace} $V(E):=Sp\{V_a\}$ of the tangent space and the {\em horizontal subspace} $H(E):=Sp\{H_a\}$ respectively, where
$$ H_a = \frac{\partial}{\partial x^a} -\Gamma^b_a\frac{\partial}{\partial u^b}.$$
 So an adapted local basis on $E$ is given by $\{\Gamma, V_a, H_a\}$ with dual basis $\{dt,\psi^a,\theta^a\}$ where $$\psi^a:=du^a-F^adt+\Gamma^a_b\theta^b.$$
Following directly from the definitions of $\Gamma,V_a$ and $H_a$ given in \cite{Cramp84,P00} we note:
$$
[\Gamma, H_a]=\Gamma^b_aH_b + \Phi^b_aV_b, \quad [\Gamma,
V_a]=-H_a + \Gamma^b_aV_b, \quad [V_a,V_b]=0,
$$

$$
[H_a, V_b]=-\frac{1}{2}\left(\frac{\partial^2 F^c}{\partial u^a\partial
u^b}\right)V_c=V_b(\Gamma^c_a)V_c=V_a(\Gamma^c_b)V_c=[H_b, V_a],
$$
and
$$
[H_a,H_b]=R^d_{ab} V_d,
$$
where the curvature of the connection is defined by
$$
R^d_{ab}:=\frac{1}{2}\left (\frac{\partial^2 F^d}{\partial
x^a\partial u^b} - \frac{\partial^2 F^d}{\partial x^b\partial u^a}
+\frac{1}{2}\left (\frac{\partial F^c}{\partial
u^a}\frac{\partial^2 F^d}{\partial u^c\partial u^b} -
\frac{\partial F^c}{\partial u^b}\frac{\partial^2 F^d}{\partial
u^c\partial u^a} \right )\right ).
$$

Following~\cite{SVCM95} we briefly introduce vector fields and forms {\em along the
projection} $\pi^0_1:E\to\mathbb R \times M$.  Vector fields along $\pi^0_1$ are sections of the pull back bundle ${\pi^0_1}^*(T(\mathbb R \times M))$ over
$E$. $\mathfrak{X}(\pi^0_1)$ denotes the $C^\infty(E)$ module of
such vector fields. In the current context see~\cite{JP01}.

With $X:= X^0\frac{\partial}{\partial t} + X^a\fpd{}{x^a}
$ and $dt, \theta^a \in \mathfrak{X}^*(\pi^0_1)$  we define the following lifts from $\mathfrak{X}(\pi^0_1)$ to $\mathfrak{X}(E)$:
\begin{equation*}
X^\Gamma :=dt(X)\Gamma,\quad
X^H :=\theta^a(X)H_a,\quad
X^V :=\theta^a(X)V_a.
\end{equation*}
For an intrinsic treatment of these various lifts see \cite{KP08}.

Any vector $W \in \mathfrak{X}(E)$ can then be uniquely decomposed as
\begin{equation*}
W=(W_\Gamma)^\Gamma+(W_H)^H+(W_V)^V
\end{equation*}
where $W_\Gamma, W_H, W_V \in \mathfrak{X}(\pi^0_1)$ with
\begin{align*}
W_\Gamma &:=dt(W)(\frac{\partial}{\partial t}+u^a\frac{\partial}{\partial x^a}),\\
W_H &:=dt(W)\frac{\partial}{\partial t}+dx^a(W)\frac{\partial}{\partial x^a},\quad
W_V :=\psi^a(W)\frac{\partial}{\partial x^a}.
\end{align*}

The {\em dynamical covariant derivative} $\nabla$ and the {\em Jacobi endomorphism} $\Phi$ that appear in \eqref{Helmholtz1} are defined along the tangent bundle projection through the following formulas
\[
[\Gamma,\V{X}]= -\H{X} + \V{(\del X)} , \qquad [\Gamma,\H{X}]=
\H{(\del X)} + \V{\Phi(X)}.
\]
In coordinates,
\[
\Phi = \Phi^a_b \fpd{}{x^a}\otimes \theta^b,
\]
with $\Phi^a_b$ as defined before. $\del$ can be extended to act on forms along the projection and is given in coordinates by
\[
\del f=\Gamma(f)\ \ \mbox{on functions}, \qquad
\del\fpd{}{x^a}=\Gamma^b_a\fpd{}{x^b},\qquad
\del\theta^a=-\Gamma^a_b\theta^b.
\]
Massa and Pagani \cite{MaPa94} introduced a linear connection on $E$ by imposing some natural requirements. If we denote this connection by $\hat{\nabla}$, these are that the covariant differentials $\hat{\nabla}dt, \hat{\nabla}S$ and $\hat{\nabla}\Gamma$ are all zero and that the vertical sub-bundle is flat. In coordinates, they are:

\begin{align}
\notag &\hnabla_\Gamma\Gamma = 0 , &  &\hnabla_\Gamma H_a = \Gamma^b_a H_b,     &  &\hnabla_\Gamma V_a  = \Gamma^b_a V_b,\\
\label{covariant-deriv-defn-basis} &\hnabla_{H_a}\Gamma = 0,  & &\hnabla_{H_a}H_b = \frac{\partial \Gamma^c_{a}}{\partial u^b} H_c,  &  &\hnabla_{H_a}V_b = \frac{\partial \Gamma^c_{a}}{\partial u^b} V_c, \\
\notag &\hnabla_{V_a}\Gamma = 0,  & &\hnabla_{V_a}H_b = 0,                  &  &\hnabla_{V_a}V_b = 0,
\end{align}

With any linear connection, $\tilde \nabla$ (not to be confused with the dynamical covariant derivative), on a manifold $M$ there is an associated shape map and torsion (see \cite{JP01}).
The shape map $A_Z: \mathfrak{X}(M) \rightarrow \mathfrak{X}(M)$, as given in Jerie and Prince~\cite{JP01}, is defined as
$$
A_Z(\xi) := \frac{d}{dt}\Bigl|_{t=0} \tau_t^{-1}(\zeta_{t*}\xi), \qquad \text{where }\xi \in T_xM.
$$
where $\tau_t: T_xM \rightarrow T_{\zeta_t(x)}M$ is the parallel transport map, defining along the flow $\{\zeta_t\}$.\\
More useful representations of this shape map on $M$ are
\begin{align*}
A_X Y &= \tilde \nabla_X Y - [X,Y],\\ 
A_XY &= \tilde \nabla_YX + T(X,Y)\notag
\end{align*}
where $X,Y \in \mathfrak{X}(M)$ and the torsion is defined as usual by
\begin{align*}
T(X,Y):=\tilde \nabla_XY-\tilde \nabla_YX-[X,Y].
\end{align*}

The torsion, $\hat T$, of the Massa and Pagani connection captures all the important properties of the SODE $\Gamma$:
\begin{align*}
&\hat T(\Gamma,V_a)=H_a,\ \hat T(\Gamma,H_a)=-\Phi^b_aV_b,\ \hat T(V_a,V_b)=0,\\ &\hat T(V_a,H_b)=0,\ \hat T(H_a,H_b)=-R^c_{ab}V_c.
\end{align*}

The commutators of the vector fields on $E$ can be written in terms of connection $\hnabla$ and the shape map as
\begin{subequations}\label{CD-VF-commutators}
\begin{align}
& [\Gamma, X^V] = \hnabla_\Gamma X^V - A_\Gamma(X^V),  \label{CD-VF-commutators-1}\\
& [\Gamma, X^H] =\hnabla_\Gamma X^H - A_\Gamma (X^H),   \label{CD-VF-commutators-2}\\
&  [X^V, Y^V] = \hnabla_{X^V} Y^V - \hnabla_{Y^V} X^V,    \label{CD-VF-commutators-3}\\
& [X^V, Y^H] = \hnabla_{X^V} Y^H - \hnabla_{Y^H} X^V,    \label{CD-VF-commutators-4}\\
& [X^H, Y^H] = \hnabla_{X^H} Y^H - \hnabla_{Y^H} X^H + R(X^H,Y^H).  \label{CD-VF-commutators-5}
\end{align}
\end{subequations}
It can be seen that $\hnabla_\Gamma$ and $A_\Gamma$ replace the role of $\nabla$ and $\Phi$ respectively.

The curvature tensor $R$ in \eqref{CD-VF-commutators-5} is a (1,2) tensor field on $E$ defined as follows $$R:=R^c_{ab}(\theta^a \wedge \theta^b)\otimes V_c.$$
We will make no notational distinction between $\Phi$ acting along the tangent bundle projection and acting as an endomorphism on $E$ as $\Phi=\Phi^b_a \ V_b\otimes\theta^a.$

In the next section we briefly sketch the ideas of the exterior
differential systems approach, specifically in the context of the
inverse problem.

\section{Helmholtz conditions and the EDS process}

\subsection{The Helmholtz conditions}
For a given regular Lagrangian $L$, there is a unique vector field, called the {\em Euler-Lagrange field},
$\Gamma$ on
 $E$ such that
$$
\Gamma \hook d\theta_L=0 \quad \text{and} \quad d t(\Gamma)=1.
$$
where $\theta_L$ is the {\em Poincar\'{e}-Cartan 1-form},
$$\theta_L:=Ldt+dL\circ S=Ldt+\frac{\partial L}{\partial u^a}\theta^a.$$
This vector field is a SODE, and the equations satisfied by its integral
curves are the Euler-Lagrange equations for $L$. By careful observation of the properties of the Cartan 2-form $d\theta_L$
the following theorem from~\cite{Cramp84} gives a transparent geometric version of the Helmholtz conditions.
\begin{thm} \label{Helmholtz2}\cite{Cramp84}
Given a SODE $\Gamma$, the necessary and sufficient conditions for there to be Lagrangian, whose Euler-Lagrange field is $\Gamma$, are that there should exist a 2-form $\Omega$ satisfying
\begin{align}
& \Omega \text{ is of maximal rank, i.e.,}\ \wedge^n\Omega\neq 0, \label{helmholtz2-1}\\
&\Omega(V_a,V_b)=0, \qquad \forall\ V_a,V_b \in V(E), \label{helmholtz2-2}\\
&\Gamma \hook \Omega=0, \label{helmholtz2-3}\\
&d\Omega=0. \label{helmholtz2-4}
\end{align}
\end{thm}
We will briefly show how the Helmholtz conditions in \eqref{Helmholtz1} arise from this theorem as follows.\\
In Crampin, Prince and Thompson \cite{Cramp84} it is shown by using $\mathcal{L}_r\Omega=0$ (from \eqref{helmholtz2-3} and \eqref{helmholtz2-4}) that    
\begin{align}\label{2-form-Omega}
\Omega=g_{ab}\psi^a\wedge\theta^b, \quad |g_{ab}|\neq 0.
\end{align}
The condition \eqref{helmholtz2-4} gives
\begin{align}\label{closure-condition}
 d\Omega(X,Y,X)=0 \text{ for all } X,Y,Z \in \{\Gamma, V_a, H_b\}.
\end{align}

Applying the formula for the exterior derivative of a 2-form $\Omega$,
\begin{align}\label{exter-deri-2-form}
d\Omega(X,Y,Z)=\sum_{cyclic\ X,Y,Z} \left(X(\Omega(Y,Z))-\Omega([X,Y],Z)\right),
\end{align}
to \eqref{closure-condition} we find that $d\Omega(V_a,V_b,V_c)=0$ is trivial and\\
\begin{align}
\notag &d\Omega(\Gamma, V_a, V_b)=0&&\Leftrightarrow && g_{ab}=g_{ba},\hfill \\
\label{closure-condition-1}&d\Omega(\Gamma, H_a, H_b)=0 &&\Leftrightarrow  && g_{ac}\Phi^c_b=g_{bc}\Phi^c_a,\hfill \\
\notag &d\Omega(\Gamma, V_a, H_b)=0 &&\Leftrightarrow  && \Gamma(g_{ab})-g_{ac}\Gamma^c_b-g_{bc}\Gamma^c_a=0,\hfill \\
\notag &d\Omega(V_a, V_b, H_c)=0 &&\Leftrightarrow && \frac{\partial g_{ab}}{\partial\dot x^c}=\frac{\partial g_{ac}}{\partial\dot x^b}.\hfill
\end{align}
Therefore $d\Omega = 0$ is equivalent to the Helmholtz conditions in \eqref{Helmholtz1} since the two remaining conditions in \eqref{closure-condition} are the consequences of the four conditions in \eqref{closure-condition-1} as shown in this context by Aldridge \cite{Al03}.
In this paper we concentrate on the case where the matrix representation, $\bold \Phi=(\Phi^a_b)$, of $\Phi$ is diagonalisable. As in \cite{CSMBP94} and \cite{STP} this corresponds to Douglas' case I, case IIa or case III. Our choice of the basis for $\mathfrak{X}(E)$ to work with is $\{\Gamma, X_a^V, X_a^H\}$, where $X_a^V$ and $X_a^H$ are vertical lifts and horizontal lifts of eigenvectors $X_a$ of diagonalisable $\bold \Phi$ on the tangent bundle projection. Let $X_a=X_a^b\frac{\partial}{\partial x^b} \in \mathfrak{X}(\pi^0_1)$ be eigenvectors of $\bold \Phi$ corresponding to eigenvalues $\lambda_a$, then their vertical lifts and their horizontal lifts are $X_a^V:=X_a^bV_b$ and $X_a^H:=X_a^b H_b$ respectively. The lifted eigenforms $\phi^{aV}$ and $\phi^{aH}$ of eigenforms $\phi^a$ together with $dt$ form the dual basis $\{dt, \phi^{aV}, \phi^{aH}\}$ to the basis $\{\Gamma, X_a^V, X_a^H\}$. This means that we will look for a non-degenerate closed 2-form $\omega \in \Sigma:=Sp\{\phi^{aV}\wedge \phi^{bH}\}$, that is, $\omega=r_{ab}\phi^{aV}\wedge \phi^{bH}$, instead of the one in \eqref{2-form-Omega}. Once the $r_{ab}$'s are found, then the multiplier $g_{ab}$ is given by
\begin{equation}\label{conver-r-to-g}
g_{ab}=r_{cd}\phi^c_a\phi^d_b,
\end{equation}
where $\phi^c_a$ and $\phi^d_b$ are components of eigenforms $\phi^c$ and $\phi^d$ respectively.

The commutators identities regarding the basis $\{\Gamma, X_a^V, X_a^H\}$ derived from \eqref{CD-VF-commutators-1}-\eqref{CD-VF-commutators-5} are as follows
\begin{subequations} \label{bracket-iden}
\begin{align}
\label{bracket-iden-a}[\Gamma,X_b^V] &=\tau_b^{a\Gamma}X_a^V-X_b^H,\\
\label{bracket-iden-b}[\Gamma,X_b^H] &=\tau_b^{a\Gamma}X_a^H+\lambda_b\delta_b^aX_a^V=\tau_b^{a\Gamma}X_a^H+\lambda_bX_b^V,\\
\label{bracket-iden-c}[X_b^V,X_c^V]  &=(\tau_{bc}^{aV}-\tau_{cb}^{aV})X_a^V,\\
\label{bracket-iden-d}[X_b^V,X_c^H]  &=\tau_{bc}^{aV}X_a^H-\tau_{cb}^{aH}X_a^V,\\
\label{bracket-iden-e}[X_b^H,X_c^H]  &=(\tau_{bc}^{aH}-\tau_{cb}^{aH})X_a^H+\phi^{aV}(R(X_b^H,X_c^H))X_a^V,
\end{align}
\end{subequations}
where the $\tau$'s are defined through
\begin{align}
\notag &\hnabla_\Gamma X_b^V = \tau_b^{a\Gamma}X_a^V , &  &\hnabla_\Gamma X_b^H= \tau_b^{a\Gamma}X_a^H, \\
\label{tau-defn} &\hnabla_{X_b^V}X_c^V=\tau^{aV}_{bc}X_a^V,  & &\hnabla_{X_b^V}X_c^H=\tau^{aV}_{bc}X_a^H, \\
\notag &\hnabla_{X_b^H}X_c^V=\tau^{aH}_{bc}X_a^V,  & &\hnabla_{X_b^H}X_c^H=\tau^{aH}_{bc}X_a^H.
\end{align}
We also have
\begin{align*}
A_\Gamma(X_b^V)&= A_\Gamma(X_b^a V_a)= X_b^a H_a =X_b^H,\\
A_\Gamma(X_b^H)&= A_\Gamma(X_b^a H_a)= -X_b^a\phi^c_a V_c =-\lambda_b X_b^c V_c=-\lambda_b X_b^V.
\end{align*}

We will now calculate the exterior derivatives of the eigenforms $\phi^{aV}$ and $\phi^{aH}$ by using the identity $d\phi(X,Y)=X(\phi(Y))-Y(\phi(X))-\phi([X,Y])$ and the bracket identities \eqref{bracket-iden-a}-\eqref{bracket-iden-e} as follows
\begin{align*}
 d\phi^{aV}(\Gamma, X_b^V) &= -\phi^{aV}(-X_b^H+\tau_b^{d\Gamma}X_d^V)=-\tau_b^{a\Gamma},\\
 d\phi^{aV}(\Gamma, X_b^H) &= -\phi^{aV}(\tau_b^{d\Gamma}X_d^H+\lambda_bX_b^V)=-\lambda_b \delta ^a_b,\\
 d\phi^{aV}(X_b^V,X_c^V)   &= -\phi^{aV}((\tau_{bc}^{dV}-\tau_{cb}^{dV})X_d^V)=\tau_{cb}^{aV}-\tau_{bc}^{aV},\\
 d\phi^{aV}(X_b^V,X_c^H)   &= -\phi^{aV}(\tau_{bc}^{dV}X_d^H-\tau_{cb}^{dH}X_d^V)=\tau_{cb}^{aH},\\
 d\phi^{aV}(X_b^H,X_c^H)   &= -\phi^{aV}((\tau_{bc}^{dH}-\tau_{cb}^{dH})X_d^H+\phi^{dV}(R(X_b^H,X_c^H))X_d^V)\\
                           &= -\phi^{aV}(R(X_b^H, X_c^H)).
\end{align*}


And similarly for $ d\phi^{aH}$:
\begin{align*}
 d\phi^{aH}(\Gamma, X_b^V) &=\delta ^a_b, \quad
 d\phi^{aH}(\Gamma, X_b^H) =-\tau_b^{a\Gamma}, \\
 d\phi^{aH}(X_b^V,X_c^V)   &= 0, \quad d\phi^{aH}(X_b^V,X_c^H)=-\tau_{bc}^{aV},\\
 d\phi^{aH}(X_b^H,X_c^H)   &= \tau_{cb}^{aH}-\tau_{bc}^{aH}.
\end{align*}
Putting these components together then we get:
\begin{align}
 d\phi^{aV} = & -\tau^{a\Gamma}_b dt \wedge \phi^{bV} - \lambda_a dt \wedge \phi^{aH} + \tau_{cb}^{aH} \phi^{bV} \wedge \phi^{cH} + \tau_{cb}^{aV} \phi^{bV} \wedge \phi^{cV} \label{dphiaV}\\
  \notag            & - \frac{1}{2}\phi^{aV}(R(X_b^H, X_c^H)) \phi^{bH} \wedge \phi^{cH},\\
 d\phi^{aH} = & \ dt \wedge \phi^{aV} - \tau^{a\Gamma}_{b} dt \wedge \phi^{bH} + \tau_{cb}^{aH} \phi^{bH} \wedge \phi^{cH} - \tau_{bc}^{aV} \phi^{bV} \wedge \phi^{cH}, \label{dphiaH}
\end{align}

We will use the following equivalent two Frobenius integrability conditions on a co-distribution $D_a^\bot=Sp\{\phi^{aV},\phi^{aH}\}$ of (lifted) eigen-forms of $\bold \Phi.$ We use the fact that $\omega^a:=\phi^{aV}\wedge\phi^{aH}$ is a characterising form for $D_a^\bot$.
\begin{align}
&d\phi^{aV} \equiv 0 \ \text{and } d\phi^{aH} \equiv 0 \ (\text{mod } \phi^{aV}, \phi^{aH}), \label{comp-in-cond-1}\\
\text{equivalently,} \ &d\omega^a=\xi^a_a\wedge \omega^a \quad (\text{no sum on $a$}),\ \xi^a_a \in \textstyle{\bigwedge^1(E)}\ \text{i.e.,}\ d\omega^a \equiv 0 \ (\text{mod } \omega^{a}) \label{comp-in-cond-2}.
\end{align}
We note here that for the one-form $\xi^a_a$ as given in \eqref{comp-in-cond-2},
\begin{equation}\label{dxi-aa}
d\xi^a_a \equiv 0 \quad (\text{mod } \phi^{aV}, \phi^{aH}).
\end{equation}

By looking at the expressions for $d\phi^{aV}$ from \eqref{dphiaV} and $d\phi^{aH}$ from \eqref{dphiaH}, together with \eqref{comp-in-cond-1}  we get the following proposition.

\begin{propn}\label{comp-in-cond}
The necessary and sufficient conditions for an eigen co-distribution $D_a^\bot=Sp\{\phi^{aV},\phi^{aH}\}$ of $\bold \Phi$ to be (Frobenius) integrable are:
\begin{equation*}
\tau^{a\Gamma}_b=0,\ \tau^{aV}_{bc}=0,\ \tau^{aH}_{bc}=0,\ \phi^{aV}(R(X_b^H,X_c^H))=0
\end{equation*}
for all $b, c \neq a.$
\end{propn}

\begin{defn}\label{ID-def}
Let $D_a^\bot:=Sp\{\phi^{aV},\phi^{aH}\}$ be an eigen co-distribution of $\bold \Phi$. A 1-form $\alpha\in D_a^\bot$ is said to be {\em an integrable direction in $D_a^\bot$} if
\begin{align}
d\alpha=\kappa\wedge\alpha,\label{integrable-direction-def-1}
\end{align}
for some 1-form $\kappa.$
\end{defn}
It can be seen from the expressions of the exterior derivatives of $\phi^{aV}$ in \eqref{dphiaV} and $\phi^{aH}$ in \eqref{dphiaH} that the 1-forms $\phi^{aV}$ may be integrable but $\phi^{aH}$ can't be. Thus we can express an integrable direction $\alpha_a$ in $D_a^\bot$ as $\alpha_a=\phi^{aV}+B_a\phi^{aH}$.

\begin{propn}\label{ID-cond}
The necessary and sufficient conditions for the existence of an integrable direction $\alpha_a=\phi^{aV}+B_a\phi^{aH}$, that is $d\alpha_a=\kappa_a\wedge\alpha_a$, in an eigen co-distribution $D_a^\bot=Sp\{\phi^{aV},\phi^{aH}\}$ of $\bold \Phi$ are, for $b,c \neq a$:
\begin{align}
&\tau^{a\Gamma}_b=0,  \label{integrable-direction-cond-1}\\
&\tau^{aH}_{cb}-B_a\tau^{aV}_{bc}=0, \label{integrable-direction-cond-2}\\
&\tau^{aV}_{cb}=\tau^{aV}_{bc}, \label{integrable-direction-cond-3}\\
&\phi^{aV}(R(X_b^H,X_c^H))+B_a(\tau^{aH}_{cb}-\tau^{aH}_{bc})=0, \label{integrable-direction-cond-4}\\
&\Gamma(B_a)=(B_a)^2+\lambda_a, \label{integrable-direction-cond-5}\\
&X_b^V(B_a)=B_a\tau^{aV}_{ab}-\tau^{aH}_{ab}, \label{integrable-direction-cond-6}\\
&X_b^H(B_a)=\phi^{aV}(R(X_b^H,X_a^H))+B_a(B_a\tau^{aV}_{ab}-\tau^{aH}_{ab}), \label{integrable-direction-cond-7}
\end{align}
and $\kappa_a$ is given by
\begin{align}
\label{ID-kappaa}\kappa_a=&(\tau^{a\Gamma}_a+B_a) dt +(\tau^{aV}_{ab}-\tau^{aV}_{ba})\phi^{bV}+(B_a\tau^{aV}_{ab}-\tau^{aH}_{ba})\phi^{bH}\\
\notag &\quad +(B_a\tau^{aV}_{aa}-X_a^V(B_a)-\tau^{aH}_{aa})\phi^{aH}.
\end{align}
\end{propn}
\begin{proof}
Let $\alpha_a=\phi^{aV}+B_a\phi^{aH} \in Sp\{\phi^{aV},\phi^{aH}\}$ be an integrable direction. By definition \ref{ID-def},
\begin{align*}
d\alpha_a=\kappa_a \wedge\alpha_a.
\end{align*}
We apply the formula for the exterior derivative of a 1-form:
\begin{align*}
d\alpha_a(X,Y)=X(\alpha_a(Y))-Y(\alpha_a(X))-\alpha_a([X,Y])
\end{align*}
to the basis of vector fields $\{\Gamma, X_a^V,X_a^H-B_aX_a^V, X_b^V, X_b^H ; \forall b \neq a\}$. $d\alpha_a(X,Y)$ is zero for basis pairs $(X,Y)$ when neither is $X_a^V$. This gives, for $b,c \neq a$,
\begin{align*}
&0=d\alpha_a(\Gamma, X_b^V)=-\tau^{a\Gamma}_b,\\
&0=d\alpha_a(\Gamma, X_b^H)=B_a\tau^{a\Gamma}_b,\\
&0=d\alpha_a(X_b^V, X_c^V)=\tau^{aV}_{bc}-\tau^{aV}_{cb},\\
&0=d\alpha_a(X_b^H, X_c^V)=B_a\tau^{aV}_{cb}-\tau^{aH}_{bc},\\
&0=d\alpha_a(X_b^H, X_c^H)=-(B_a(\tau^{aH}_{bc}-\tau^{aH}(B_a)^2,\\
&0=d\alpha_a(X_b^V, X_a^H-B_a X_a^V)=X_b^V(B_a)-B_a\tau^{aV}_{ab}+\tau^{aH}_{ab}),\\
&0=d\alpha_a(X_b^H, X_a^H-B_a X_a^V)=X_b^H(B_a)-B_a(B_a\tau^{aV}_{ab}-\tau^{aH}_{ab})-\phi^{aV}(R(X_b^H,X_a^H)).
\end{align*}
These immediately give the conditions \eqref{integrable-direction-cond-1}-\eqref{integrable-direction-cond-7}.
For the remainder we have:
\begin{align*}
&d\alpha_a(\Gamma, X_a^V)=\tau^{a\Gamma}_a+B_a,\\
&d\alpha_a(X_b^V, X_a^V)=\tau^{aV}_{ab}-\tau^{aV}_{ba},\\
&d\alpha_a(X_b^H, X_a^V)=B_a\tau^{aV}_{ab}-\tau^{aH}_{ba},\\
&d\alpha_a(X_a^H-B_aX_a^V, X_a^V)=-X_a^V(B_a)+B_a\tau^{aV}_{aa}-\tau^{aH}_{aa}.
\end{align*}
These give the formula for $\kappa_a$ as in \eqref{ID-kappaa}.\\
Conversely, if conditions \eqref{integrable-direction-cond-1}-\eqref{integrable-direction-cond-7} hold, then $d\alpha_a(X,Y)=0$ for all basis pairs $(X,Y)$ when neither is $X_a^V$ and thus $d\alpha_a=\kappa_a \wedge \alpha_a$ as expected. 
\end{proof}
\subsection{Jacobi identities in the inverse problem}

The $\tau$'s defined in \eqref{tau-defn} are not independent and we now examine their relations. To this end we apply the Jacobi identity
$$[X,[Y,Z]]+[Y,[Z,X]]+[Z,[X,Y]]=0$$
to the basis $\{\Gamma, X_a^V, X_b^H\}$. This results in the following identities which are very useful for later calculations.\\ 
(We remark in passing that the Bianchi identities for $\hat\nabla$ are redundant in the presence of the Jacobi identities ( see \cite{DP14}) and we do not consider them here.)

Applying this to the triple $\Gamma, X_b^V$ and $X_c^H$ we have,\\
\begin{equation*} [\Gamma,[X_b^V,X_c^H]]+[X_b^V,[X_c^H,\Gamma]]+[X_c^H,[\Gamma,X_b^V]]=0.
\end{equation*}
This gives following identities
\begin{subequations} \label{bracket-iden-1}
\begin{align}
\label{bracket-iden-1-a} \Gamma(\tau_{bc}^{aV}) &=-\tau_{bc}^{eV}\tau_e^{a\Gamma}-\tau_{bc}^{aH}+
                                         X_b^V(\tau_c^{a\Gamma})+\tau_c^{e\Gamma}\tau_{be}^{aV}+\tau_{ec}^{aV}\tau_b^{e\Gamma},\\
\label{bracket-iden-1-b} \Gamma(\tau_{cb}^{aH})&=-\tau_{cb}^{eH}\tau_e^{a\Gamma}+\tau_c^{e\Gamma}\tau_{eb}^{aH}-X_c^H(\tau_b^{a\Gamma})+\tau_b^{e\Gamma}\tau_{ce}^{aH}-X_b^V(\lambda_c)\delta^a_c\\
\notag                                         &+\lambda_c(\tau_{cb}^{aV}-\tau_{bc}^{aV})+\lambda_a\tau^{aV}_{bc}+\phi^{aV}(R(X_b^H,X_c^H)).
\end{align}
\end{subequations}

Applying the Jacobi identity to the triple $(\Gamma, X_b^V,X_c^V)$ gives no new identities since the expression that we get from this is
$$\Gamma(\tau_{bc}^{aV}-\tau_{cb}^{aV})=\tau_e^{a\Gamma}(\tau_{cb}^{eV}-\tau_{bc}^{eV})+\tau_{cb}^{aH}-\tau_{bc}^{aH}+
                                         X_b^V(\tau_c^{a\Gamma})-X_c^V(\tau_b^{a\Gamma})+\tau_c^{e\Gamma}(\tau_{be}^{aV}-\tau_{eb}^{aV})-\tau_b^{e\Gamma}(\tau_{ce}^{aV}-\tau_{ec}^{aV})$$

which can be obtained by using \eqref{bracket-iden-1-a}.\\
Applying the Jacobi identity to other triples we get further identities:\\
for $(\Gamma, X_b^H,X_c^H)$ we obtain
\begin{subequations}
\begin{align}
\label{bracket-iden-2-a} \Gamma(\phi^{aV}(R(X_b^H,X_c^H)))&=\tau_c^{e\Gamma}\phi^{aV}(R(X_b^H,X_e^H))+\tau_b^{e\Gamma}\phi^{aV}(R(X_e^H,X_c^H))\\
\notag                                                   &-\tau_e^{a\Gamma}\phi^{eV}(R(X_b^H,X_c^H))+(\lambda_c-\lambda_a)\tau_{bc}^{aH}\\
\notag                                                   &-(\lambda_b-\lambda_a)\tau_{cb}^{aH}+X_b^H(\lambda_c)\delta^a_c-X_C^H(\lambda_b)\delta^a_b, \\
\label{bracket-iden-2-b} \Gamma(\tau_{bc}^{aH}-\tau_{cb}^{aH})&=-(\tau_{bc}^{eH}-\tau_{cb}^{eH})\tau_e^{a\Gamma}+\tau_c^{e\Gamma}(\tau_{be}^{aH}-\tau_{eb}^{aH})-\lambda_c\tau_{cb}^{aV}\\
\notag                                                       &+\lambda_b\tau_{bc}^{aV})+ X_c^H(\tau_b^{a\Gamma})-X_b^H(\tau_c^{a\Gamma})\\
\notag                                                       &-\tau_b^{e\Gamma}(\tau_{ce}^{aH}-\tau_{ec}^{aH})+\phi^{aV}(R(X_b^H,X_c^H)).
\end{align}
\end{subequations}
Substituting $\Gamma(\tau_{bc}^{aH})$ and $\Gamma(\tau_{cb}^{aH})$ from \eqref{bracket-iden-1-b} into \eqref{bracket-iden-2-b} we have
\begin{align} \label{bracket-iden-3}
3\phi^{aV}(R(X_b^H,X_c^H))=X_b^V(\lambda_c)\delta^a_c-X_c^V(\lambda_b)\delta^a_b+\tau^{aV}_{bc}(\lambda_c-\lambda_a)-\tau^{aV}_{cb}(\lambda_b-\lambda_a).
\end{align}
For the triple $(X_a^V, X_b^V,X_c^H)$ we have
\begin{subequations} \label{bracket-iden-3}
\begin{align}
\label{bracket-iden-3-a} X_a^V(\tau_{bc}^{dV})-X_b^V(\tau_{ac}^{dV})&=\tau_{ac}^{eV}\tau_{be}^{dV}-\tau_{bc}^{eV}\tau_{ae}^{dV}+\tau_{ec}^{dV}(\tau_{ab}^{eV}-\tau_{ba}^{eV}),\\
\label{bracket-iden-3-b} X_a^V(\tau_{cb}^{dH})-X_b^V(\tau_{ca}^{dH})&=X_c^H(\tau_{ab}^{dV}-\tau_{ba}^{dV})-\tau_{bc}^{eV}\tau_{ea}^{dH}+\tau_{ac}^{eV}\tau_{eb}^{dH}\\
\notag                                                               &+\tau_{ce}^{dH}(\tau_{ab}^{eV}-\tau_{ba}^{eV})+\tau_{ca}^{eH}(\tau_{be}^{dV}-\tau_{eb}^{dV})\\
\notag                                                               &-\tau_{cb}^{eH}(\tau_{ae}^{dV}-\tau_{ea}^{dV}).
\end{align}
\end{subequations}
Applying the Jacobi identity to $(X_a^H, X_b^V,X_c^H)$ we get
\begin{subequations} \label{bracket-iden-4}
\begin{align}
\label{bracket-iden-4-a} X_a^H(\tau_{bc}^{dV})-X_c^H(\tau_{ba}^{dV})&=X_b^V(\tau_{ac}^{dH}-\tau_{ca}^{dH})-\tau_{bc}^{eV}(\tau_{ae}^{dH}-\tau_{ea}^{dH})\\
\notag                                                                   &-\tau_{cb}^{eH}\tau_{ea}^{dV}-(\tau_{ca}^{eH}-\tau_{ac}^{eH})\tau^{dV}_{be}+\tau^{eV}_{ba}(\tau^{dH}_{ce}-\tau^{dH}_{ec})+\tau^{eH}_{ab}\tau^{dV}_{ec},\\
\label{bracket-iden-4-b} X_a^H(\tau_{cb}^{dH})-X_c^H(\tau_{ab}^{dH})&=X_b^V(\phi^{dV}(R(X_c^H,X_a^H)))+\tau_{bc}^{eV}\phi^{dV}(R(X_a^H,X_e^H))\\
\notag                                                              &-\tau_{cb}^{eH}\tau_{ae}^{dH}+\phi^{eV}(R(X_c^H,X_a^H))(\tau_{be}^{dV}-\tau_{eb}^{dV})\\
\notag                                                              &+\tau_{ab}^{eH}\tau_{ce}^{dH}-\tau_{eb}^{dH}(\tau_{ca}^{eH}-\tau_{ac}^{eH})-\tau_{ba}^{eV}\phi^{dV}(R(X_c^H,X_e^H)).
\end{align}
\end{subequations}

For the triple $(X_a^V, X_b^V,X_c^V)$ we get only one new identity
\begin{align} \label{bracket-iden-5}
&X_a^V(\tau^{dV}_{bc}-\tau^{dV}_{cb})+X_b^V(\tau^{dV}_{ca}-\tau^{dV}_{ac})+X_c^V(\tau^{dV}_{ab}-\tau^{dV}_{ba})\\
\notag &=-(\tau^{dV}_{ae}-\tau^{dV}_{ea})(\tau^{eV}_{bc}-\tau^{eV}_{cb})-(\tau^{dV}_{be}-\tau^{dV}_{eb})(\tau^{eV}_{ca}-\tau^{eV}_{ac})\\
\notag  &-(\tau^{dV}_{ce}-\tau^{dV}_{ec})(\tau^{eV}_{ab}-\tau^{eV}_{ba}).
\end{align}
Finally, applying the Jacobi identity to $(X_a^H, X_b^H,X_c^H),$ we have two more identities
\begin{subequations}\label{bracket-iden-6}
\begin{align}
\label{bracket-iden-6-a} &X_a^H(\tau^{dH}_{bc}-\tau^{dH}_{cb})+X_b^H(\tau^{dH}_{ca}-\tau^{dH}_{ac})+X_c^H(\tau^{dH}_{ab}-\tau^{dH}_{ba})\\
\notag &=\tau_{ea}^{dV}\phi^{eV}(R(X_b^H,X_c^H))+\tau_{eb}^{dV}\phi^{eV}(R(X_c^H,X_a^H))\\
\notag &+\tau_{ec}^{dV}\phi^{eV}(R(X_a^H,X_b^H))-(\tau^{dH}_{ae}-\tau^{dH}_{ea})(\tau^{eH}_{bc}-\tau^{eH}_{cb})\\
\notag  &-(\tau^{dH}_{be}-\tau^{dH}_{eb})(\tau^{eH}_{ca}-\tau^{eH}_{ac})-(\tau^{dH}_{ce}-\tau^{dH}_{ec})(\tau^{eH}_{ab}-\tau^{eH}_{ba}),\\
\label{bracket-iden-6-b} &X_a^H(\phi^{dV}(R(X_b^H,X_c^H)))+X_b^H(\phi^{dV}(R(X_c^H,X_a^H)))\\
\notag                   &+X_c^H(\phi^{dV}(R(X_a^H,X_b^H)))=(\tau^{eH}_{cb}-\tau^{eH}_{bc})\phi^{dV}(R(X_a^H,X_e^H))\\
\notag                   &+(\tau^{eH}_{ac}-\tau^{eH}_{ca})\phi^{dV}(R(X_b^H,X_e^H))+(\tau^{eH}_{ba}-\tau^{eH}_{ab})\phi^{dV}(R(X_c^H,X_e^H))\\
\notag                   &-\tau^{dH}_{ae}\phi^{eV}(R(X_b^H,X_c^H))-\tau^{dH}_{be}\phi^{eV}(R(X_c^H,X_a^H))-\tau^{dH}_{ce}\phi^{eV}(R(X_a^H,X_b^H)).
\end{align}
\end{subequations}

\subsection{EDS and the inverse problem}\label{EDS-inverse}       
The idea of looking for closed 2-forms leads to the use of EDS method for the inverse problem. The EDS references are the book \cite{Bryant91} in general and the memoir \cite{AT92} particularly for the inverse problem. In the first part of this section we give a brief description of the EDS approach to the inverse problem. The second part is devoted to significant results regarding the so-called differential ideal step of EDS process.
\subsubsection{EDS approach to the inverse problem}
According to  Anderson and Thompson in \cite{AT92}, the EDS process for the inverse problem involves three steps. We start with the submodule of 2-forms $\Sigma$. The first step, namely the differential ideal step, is to look for a final submodule, $\Sigma^f,$ of $\Sigma$ that generates a differential ideal. The second step is to create an equivalent linear Pfaffian system for the closed 2-forms, and final step is to determine the generality of the solution of the problem by using the Cartan-K\"{a}hler theorem.\\

The differential ideal step is a recursive process which produces from $\Sigma^0:=\Sigma$ a sequence of submodules $\Sigma^0\supset\Sigma^1\supset\Sigma^2\supset...$. Each stage of the process involves calculating the exterior derivative of forms belonging to some submodule $\Sigma^i$  and then checking whether these 3-forms belong to the ideal generated by that submodule, which results in a restriction on the admissible 2-forms which form a submodule $\Sigma^{i+1}$ and then the process is repeated from this submodule and so on until a final differential ideal, $\langle\Sigma^f\rangle$, is found, i.e. $\Sigma^f=\Sigma^{f+1}$, or the trivial set is reached. If it is impossible to create a maximal rank two-form at any stage during this process, the problem has no regular solution.\\

Suppose that a differential ideal generated by $\Sigma^f$ is found,
the next step in the EDS process is to express the problem of finding the closed
2-forms in $\Sigma^f$ as a Pfaffian system. We will give a brief outline of this step, see \cite{AT92}, \cite{KP08} or  \cite{Bryant91} for details.\\
 Let the differential ideal
$\langle\Sigma^f\rangle$ be generated by the set of 2-forms, not necessary simple, $\{\bar\omega^k\}, \ k \in \{1,...,d\}$, and calculate
$$
d\bar\omega^k = \bar\xi^k_h \wedge \bar\omega^h.
$$
where the $\bar\xi^i_j$'s are now known one-forms.

Since $\omega \in \Sigma^f=Sp\{\bar\omega^k\}$, $d\omega = \beta_j \wedge \bar\omega^j$,
and because we are looking for those $\omega$'s such that $d\omega = 0$ the next step is to
find all possible $d$-tuples of one forms $(\rho^A_k) = (\rho^A_1,...,\rho^A_d)$
such that $\rho^A_k \wedge \bar\omega^k = 0$. Once all the $d$-tuples $\rho^A_k \ A \in
\{1,...,e\}$ have been found, the inverse problem becomes that of finding the functions $r_k$
which satisfy the Pfaffian system of equations:

\begin{align}
\label{pfaffian1} d r_k + r_h \bar\xi^h_k + p_A \rho^A_k = 0,
\end{align}

for some arbitrary functions $p_A$. The freedom in the
choice of these $p_A$'s will be then exploited in the final part of the EDS procedure.\\

The general method for finding the solution for this problem in EDS is to define an
extended manifold $N = E \otimes \R^d \otimes
\R^e$ with co-ordinates $\{t,x^a,u^b,r_k, p_A\}, \ a,b \in \{1,...,n\}, k \in
\{1,...,d\}, A \in \{1,...,e\}$ and look for $2n+1$ dimensional submanifolds that are
sections over $E$ and on which the one forms
$$
\sigma_k := d r_k + r_h \bar\xi^h_k + p_A \rho^A_k
$$
are zero.

To find these manifolds, $\sigma_k$ are considered constraint forms for some
distribution on $N$, and the problem becomes that of looking for integral
manifolds arising from this distribution. To find these integral manifolds, we
choose a basis of forms on $N$, $\{\alpha_m, \sigma_k, \pi_A\}$ where
$\{\alpha_m\}$ are a pulled back basis for $E$, $\pi_A:= d p_A$,  and $\sigma_k$ as
defined above completes the basis.

The condition that we have sections over $E$ is that the form
$$ \alpha_1 \wedge ... \wedge \alpha_{2n+1}$$
be non-zero on the $2n+1$ dimensional integral manifolds given by the constraint
forms.

In the remainder of this section, we will give a brief outline of the process of
finding the generality of the solutions to this last problem, see \cite{AT92} or \cite{Bryant91} for details.

According to \cite{AT92}, to determine the existence and generality of the
solutions to \eqref{pfaffian1}, we calculate the exterior derivatives $d\sigma_k$
modulo the ideal generated by the forms $\sigma_k$
\begin{equation}
\label{dsig} d\sigma_k \equiv \pi_k^i \wedge \alpha_i + t^{ij}_k \alpha_i \wedge
\alpha_j \ (\text{mod } \sigma)
\end{equation}
where $\pi_k^i$ are some linear combinations of $\pi_A$.  As $d \sigma_k$ expands
with no $d p_A \wedge d p_B$ terms, the system is quasi-linear.

As we want the system to be a section over $E$, i.e $\alpha_1 \wedge ... \wedge
\alpha_{2n+1} \neq 0$ on the integral manifolds, we need to absorb all the
$\alpha_i \wedge \alpha_j$ terms into the $\pi^i_k \wedge \alpha_i$ terms. This is
done by changing the basis forms $\pi_A$ to $\bar\pi_A:= \pi_A - l^j_A \alpha_j$.
If any of the $\alpha_i \wedge \alpha_j$ terms can not be absorbed, then asking for
$d\sigma_k \equiv 0 \ (\text{mod } \sigma)$ is incompatible with the independence
condition and therefore there is no solution.

Once the $\alpha_i \wedge \alpha_j$ terms have been removed, the system
\begin{align}\label{dsigma-f}
d\sigma_k \equiv \pi_k^i \wedge \alpha_i  \ (\text{mod } \sigma)
\end{align}
is used to create the tableau $\Pi$ from which the Cartan characters, $s_1, s_2, ..., s_k,$ can be calculated allowing us to apply the Cartan test for involution.

\begin{center}
\renewcommand{\arraystretch}{1.25}
$\Pi = $
\begin{tabular}{c|c c c c}
& $\alpha_1$ & $\alpha_2$ & \dots & $\alpha_n$ \\ \hline
$\sigma_1$ & $\pi^1_1$ & $\pi^2_1$ & \dots & $\pi^n_1$ \\
$\sigma_2$ & $\pi^1_2$ & $\pi^2_2$ & \dots & $\pi^n_2$ \\
\vdots     & \vdots    & \vdots    &       & \vdots \\
$\sigma_d$ & $\pi^1_d$ & $\pi^2_d$ & \dots & $\pi^n_d$ \\
\end{tabular}
\end{center}

\noindent
The first character $s_1$ is the number of independent one
forms that can be chosen from first column of $\Pi$.  $s_2$ is the number of
independent forms in the second column that are also independent of all forms in
the first column. This is repeated for $s_3$ and onwards until all the independent
forms are exhausted. In computing the Cartan characters, the basis $\{\alpha_i\}$ is
chosen such that $s_1$ is as large as possible, and $s_2$ large as possible but
less the $s_1$, and so on.  In particular, $s_k$ must form a non-increasing
sequence of integers.

Once the Cartan characters are found, the Cartan test for involution is performed.

Let $t$ denote the number of ways in which the forms $\pi^i_k$ can be modified by using
$\bar\pi^A = \pi^A - l^{j}_A \alpha_j$, without changing \eqref{dsigma-f}. That is, $t$
is the dimension of the linear space of $e$-tuples of one forms
$(\tau_1,\tau_2,...,\tau_e)$ of the form: $\tau_A = l^j_A \alpha_j$ such that
$$
a^{iA}_k \tau_A \wedge \alpha_i = 0
$$
Then, according to Cartan, the differential system \eqref{pfaffian1} is in
involution if and only if
$$
t = s_1 + 2 s_2 + 3 s_3 + ... + k s_k.
$$

If the Cartan test fails, then it is necessary to prolong the differential system
by differentiating the original equations to obtain a new differential system on
$N_1$, the first jet bundle of local sections of $N$ over $M$, then repeating the
foregoing analysis.

Once a sequence of Cartan characters is found that passes the Cartan test, then if
$s_l$ is the last non-zero character, the general solution to the differential
system \eqref{pfaffian1} will depend on $s_l$ arbitrary functions of $l$ variables.

\subsubsection{The differential ideal step}\label{section-IP-EDS}
It follows from Theorem \ref{Helmholtz2} that for a given SODE $\Gamma$ for which the corresponding $\bold \Phi$ is diagonalisable, the closed 2-form $\omega$ that we are seeking must satisfy the algebraic conditions:
\begin{align*}
&\omega(X_a^V,X^V_b)=0,\quad &&\omega(X_a^H,X^H_b)=0 \\
&\Gamma \hook \omega=0, \quad &&\omega(X_a^V,X^H_b)=\omega(X_b^V,X_a^H).
\end{align*}
So we start the EDS process with the module $\Sigma^0:=Sp\{\omega^{ab}\},$ where $\omega^{ab}:=\frac{1}{2}(\phi^{aV} \wedge \phi^{bH}+ \phi^{bV} \wedge \phi^{aH}),\ 1\le a\le b \le n$, and look for the (final) differential ideal generated by $\Sigma^f$.\\
Consider a 2-form $\omega$ in $\Sigma^0$, that is $\omega=\sum_{a \leq b}r_{ab}\omega^{ab}$. Calculating the exterior derivative of $\omega$ using \eqref{dphiaV} and \eqref{dphiaH} gives
\begin{align}
\notag d\omega&=\sum_{a\le b} dr_{ab}\wedge \frac{1}{2}(\phi^{aV} \wedge \phi^{bH}+ \phi^{bV}\wedge \phi^{aH})\\
\notag       &\quad - \sum_{a\le b} r_{ab}\bigg[(\tau^{a\Gamma}_c dt+\tau^{aH}_{dc}\phi^{dH}+\tau^{aV}_{dc}\phi^{dV})\wedge \frac{1}{2}(\phi^{bV} \wedge \phi^{cH}+ \phi^{cV}\wedge \phi^{bH})\\
\notag &\quad -(\tau^{b\Gamma}_c dt+\tau^{bH}_{dc}\phi^{dH}+\tau^{bV}_{dc}\phi^{dV})\wedge \frac{1}{2}(\phi^{aV} \wedge \phi^{dH}+ \phi^{dV}\wedge \phi^{aH})\\
\notag &\quad +\frac{1}{2}(\lambda_b-\lambda_a)dt\wedge \phi^{aH}\wedge \phi^{bH}\\
\notag &\quad -\frac{1}{4}\phi^{aV}(R(X_d^H,X_c^H))\phi^{dH}\wedge\phi^{cH}\wedge\phi^{bH}\\
\notag & \quad -\frac{1}{4}\phi^{bV}(R(X_d^H,X_c^H))\phi^{dH}\wedge\phi^{cH}\wedge\phi^{aH}\bigg]\\
\notag \Rightarrow \ d\omega&\equiv \sum_{a\le b} r_{ab}\bigg[\frac{1}{2}(\lambda_b-\lambda_a)dt\wedge \phi^{aH}\wedge \phi^{bH}\\
\label{sigma-0-condition} &\quad -\frac{1}{4}\phi^{aV}(R(X_d^H,X_c^H))\phi^{dH}\wedge\phi^{cH}\wedge\phi^{bH}\\
 \notag     & \quad -\frac{1}{4}\phi^{bV}(R(X_d^H,X_c^H))\phi^{dH}\wedge\phi^{cH}\wedge\phi^{aH}\bigg] \ (\text{mod } \langle \Sigma^0 \rangle).
\end{align}
By looking at \eqref{sigma-0-condition} we find alternative proofs of the following two propositions (see \cite{Al06}).

\begin{propn} \label{identity thm} 
The differential ideal step finishes at $\Sigma^0$ if and only if $\bold{\Phi}$ is a function multiple of the identity.
\end{propn}
\begin{proof} From \eqref{sigma-0-condition}, the necessary and sufficient conditions for $\langle\Sigma^0\rangle$ to be a differential ideal are
\begin{equation}
r_{ab}(\lambda_b-\lambda_a)dt\wedge \phi^{aH}\wedge \phi^{bH}=0, \quad \forall a < b \quad \text{(no sum)}\label{sigma-0-condition-1}
\end{equation}
and
\begin{equation}
\sum_{a\le b} r_{ab}(\phi^{aV}(R(X_d^H,X_c^H))\phi^{dH}\wedge\phi^{cH}\wedge\phi^{bH}
+\phi^{bV}(R(X_d^H,X_c^H))\phi^{dH}\wedge\phi^{cH}\wedge\phi^{aH})=0 \label{sigma-0-condition-2}
\end{equation}
for all $r_{ab}$.
Hence if $\Sigma^0$ generates a differential ideal, then the two conditions are satisfied for all $r_{ab}$. This implies that for all $a,b$, $\lambda_a=\lambda_b$, and thus the matrix $\bold \Phi$ is multiple of the identity. Conversely, if $\bold \Phi$ is multiple of the identity, then \eqref{sigma-0-condition-1} is immediately satisfied and \eqref{sigma-0-condition-2} is satisfied via the identity \eqref{bracket-iden-3}.
\end{proof}
\begin{propn}\label{di-sigma-0}
Suppose that $\bold \Phi$ is diagonalisable with distinct eigenvalues and eigenforms $\phi^a.$ Let $\Sigma^0$ be $Sp\{\omega^{ab}\}$ and $\omega \in \Sigma^0$. Then  $\omega \in \Sigma^1$ if and only if $\omega:=\sum_{a=1}^n r_a \omega^{aa}$ and the curvature satisfies
\begin{equation}\label{curvature-condition-Sigma-1}
\sum_{\text{cyclic } dce} r_{d}\phi^{dV}(R(X_c^H, X_e^H))=0,\quad \text{for all distinct $d,c,e$, (no sum on $d$).}
\end{equation}
\end{propn}

\begin{proof}
Let $\omega \in \Sigma^0$, that is $\omega= \sum_{a\leq b}r_{ab}\omega^{ab}$, then $\omega \in \Sigma^1$ if and only if $d\omega \in \langle \Sigma^0 \rangle$. This is equivalent to the two conditions \eqref{sigma-0-condition-1} and \eqref{sigma-0-condition-2} on $r_{ab}$.

Since $\lambda_b-\lambda_a \neq 0$ by assumption, \eqref{sigma-0-condition-1} gives $r_{ab}=0$ for $a < b$. We will now use $r_a$ instead of $r_{aa}$ and $\omega^a$ instead of $\omega^{aa}$.\\
The condition \eqref{sigma-0-condition-2} becomes
\begin{equation*}
\sum_{dce}r_d(\phi^{dV}(R(X_c^H,X_e^H))\phi^{cH}\wedge\phi^{eH}\wedge\phi^{dH}=0,
\end{equation*}
for distinct $d,c$ and $e$, which is equivalent to \eqref{curvature-condition-Sigma-1}.\\
Therefore $\Sigma^1=Sp\{\omega^a:=\phi^{aV}\wedge\phi^{aH}, a=1,\dots,n\}$ satisfying the condition \eqref{curvature-condition-Sigma-1} as required.
\end{proof}

By observation from Proposition \ref{identity thm} we have that for the case where $\bold \Phi$ is diagonalisable with distinct eigenvalues we need to process the next differential ideal step to examine whether or not $\langle\Sigma^1\rangle$ is a differential ideal. However the condition \eqref{curvature-condition-Sigma-1} presents a difficulty in this checking process. Now we denote $\tilde \Sigma^1:=Sp\{\omega^a, a=1,\dots,n\}$, which does not necessarily satisfy \eqref{curvature-condition-Sigma-1}, and so $\Sigma^1\subseteq \tilde\Sigma^1\subset\Sigma^0$. These results show that for the case where $\bold \Phi$ is diagonalisable with distinct eigenvalues, $\tilde \Sigma^1$ is the more effective option with which to start the differential ideal step.

\begin{propn}\label{di-sigma-1}
Let $\bold \Phi$ be diagonalisable with distinct eigenvalues. Then the necessary and sufficient conditions for $\omega=\sum_{a} r_a\phi^{aV}\wedge\phi^{aH} \in \tilde \Sigma^1$ to have its exterior derivative in the ideal $\langle\tilde\Sigma^1\rangle$ are that, for all distinct $a,b$ and $c$ (no sum),
\begin{align}
\notag&r_a\tau^{a\Gamma}_b+r_b\tau^{b\Gamma}_a=0, \\
\notag&r_a(\tau^{aV}_{bc}-\tau^{aV}_{cb})-r_b\tau^{bV}_{ca}+r_c\tau^{cV}_{ba}=0,\\
\label{di-sigma-1-condition}&r_a(\tau^{aH}_{bc}-\tau^{aH}_{cb})-r_b\tau^{bH}_{ca}+r_c\tau^{cH}_{ba}=0,\\
\notag&r_a\phi^{aV}(R(X_c^H,X_b^H))+r_b\phi^{bV}(R(X_a^H,X_c^H))+r_c\phi^{cV}(R(X_b^H,X_a^H))=0.
\end{align}
The last of these is just \eqref{curvature-condition-Sigma-1}.
\end{propn}

\begin{proof}
Let $\omega=r_a\phi^{aV}\wedge\phi^{aH} \in \tilde\Sigma^1$. Then
\begin{equation}
d\omega \in \langle \tilde \Sigma^1\rangle \Leftrightarrow \  d\omega=\sum_k\beta_k \wedge \phi^{kV}\wedge\phi^{kH}. \label{sigma1-condition}
\end{equation}
By observation \eqref{sigma1-condition} is equivalent to
\begin{align}
\notag & d\omega(\Gamma, X_a^V, X_b^V)=0, & &d\omega(\Gamma, X_a^H, X_b^H)=0,& & d\omega(\Gamma, X_a^V, X_b^H)=0,\\
\label{sigma1-condition-g1}& d\omega(X_a^V, X_b^H, X_c^H)=0, & &d\omega(X_a^V, X_b^V, X_c^H)=0,& &d\omega(X_a^V, X_b^V, X_c^V)=0,\\
\notag &d\omega(X_a^H, X_b^H, X_c^H)=0,
\end{align}
for all distinct $a,b$ and $c$.\\
Applying the formula \eqref{exter-deri-2-form} to the identities in \eqref{sigma1-condition-g1}, we can see that only the second part that involves the Lie brackets can contribute. Using the bracket relations \eqref{bracket-iden-a}-\eqref{bracket-iden-e} in the calculation we find that the first, the second and the sixth condition in \eqref{sigma1-condition-g1} are identically satisfied. The third condition gives:
$$r_a\tau^{a\Gamma}_b+r_b\tau^{b\Gamma}_a=0, \quad a \neq b.$$
The fourth and the fifth condition respectively give
$$r_a(\tau^{aH}_{bc}-\tau^{aH}_{cb})-r_b\tau^{bH}_{ca}+r_c\tau^{cH}_{ba}=0,$$
$$r_a(\tau^{aV}_{bc}-\tau^{aV}_{cb})-r_b\tau^{bV}_{ca}+r_c\tau^{cV}_{ba}=0,$$
for all distinct $a,b$ and $c$.\\
The remaining condition is, that for all distinct $a,b,c$ and with no sum,
$$r_a\phi^{aV}(R(X_c^H,X_b^H))+r_b\phi^{bV}(R(X_a^H,X_c^H))+r_c\phi^{cV}(R(X_b^H,X_a^H))=0.$$
This last condition is simply the condition \eqref{curvature-condition-Sigma-1} in Proposition \ref{di-sigma-0}.
\end{proof}
\begin{cor}\label{firststep-condition}
For diagonalisable $\bold \Phi$ with distinct eigenvalues, the necessary and sufficient conditions for $\tilde \Sigma^1$ to generates a differential ideal are that, for all distinct $a,b$ and $c$,
\begin{equation}\label{Sigma-1-diff-ideal}
\tau^{a\Gamma}_b=0,\ \tau^{aV}_{bc}=0
\end{equation}
\end{cor}

\begin{proof}
$\tilde \Sigma^1:=Sp\{\phi^{aV}\wedge \phi^{aH}: a=1,...,n\}$ generates a differential ideal if and only if the conditions in Proposition \ref{di-sigma-1} hold for arbitrary $r_a$. This immediately gives $\tau^{a\Gamma}_b=0, \tau^{aV}_{bc}=0$, $\tau^{aH}_{bc}=0$ and $\phi^{aV}(R(X_b^H,X_c^H))=0$ for all distinct $a,b$ and $c$. But the last two conditions are the consequences of the first two conditions by Jacobi identities \eqref{bracket-iden-1-a} and \eqref{bracket-iden-1-b} for distinct $a,b$ and $c$.
\end{proof}
We note here that if we assume $\tau^{a\Gamma}_b=0$ for all $a\neq b$ so that $\hnabla_\Gamma X_a^{V/H}=\tau^{a\Gamma}_aX_a^{V/H}$, then all $\tau^{a\Gamma}_b$ can be put equal to zero by re-scaling the eigenvectors of $\bold \Phi$. Thus from now on we have that if $\tilde \Sigma^1$ is a differential ideal, then $\tau^{a\Gamma}_b=0$ for all $a,b$.

In the next differential ideal steps, we define $\tilde \Sigma^{i+1}:=\{\omega\in \tilde\Sigma^i: d\omega\in \langle\tilde\Sigma^{i}\rangle\}$. Thus $\tilde \Sigma^2$ is the submodule of 2-forms in $\tilde\Sigma^1$ which further satisfy the conditions in \eqref{di-sigma-1-condition} and so $\tilde \Sigma^2 \subseteq \Sigma^1 \subseteq \tilde\Sigma^1$. The relation between the sequences $\tilde \Sigma^1 \supset\tilde \Sigma^2 \supset \dots\supset \tilde \Sigma^p\supset\dots$ and $\Sigma^1\supset\Sigma^2 \supset \dots\supset \Sigma^p\supset\dots$ is as follows.

\begin{lem}\label{relation-Sigma-Sigmatilde}
 $\tilde \Sigma^1 \supseteq \Sigma^1\supseteq\tilde \Sigma^2 \supseteq \Sigma^2 \supseteq \dots\supseteq \tilde \Sigma^p\supseteq\Sigma^p\supseteq\dots.$
\end{lem}
\begin{proof}
We have established this for $p=1$. Suppose that it is true for $p=k$, now we will prove that it is true for $p=k+1$, that is $\tilde \Sigma^{k+1} \supseteq \Sigma^{k+1}\supseteq\tilde \Sigma^{k+2}$.
Let $\omega \in \Sigma^{k+1}$, that is $\omega \in \Sigma^k \subseteq \tilde \Sigma^k$ and $d\omega \in \langle\Sigma^k\rangle \subseteq \langle\tilde\Sigma^k\rangle$ and so $\omega \in \tilde \Sigma^{k+1}$.\\
Now let $\omega \in \tilde\Sigma^{k+2}$, that is $\omega \in \tilde\Sigma^{k+1} \subseteq \Sigma^k$ and $d\omega \in \langle\tilde\Sigma^{k+1}\rangle \subseteq \langle\Sigma^k\rangle$ and so $\omega \in \Sigma^{k+1}$.
\end{proof}

An immediate result from Lemma \ref{relation-Sigma-Sigmatilde} is
\begin{cor}
If $\tilde \Sigma^p$ generates a differential ideal but $\tilde\Sigma^{p-1}$ does not, then either $\Sigma^{p-1}$ or $\Sigma^p$ generates a differential ideal, and moreover either $\tilde \Sigma^p=\Sigma^{p-1}$ or $\tilde \Sigma^p=\Sigma^{p}$ respectively.
\end{cor}
\begin{proof}
As a result of Proposition \ref{di-sigma-1} it is true for $p=1$, that is, if $\langle\tilde\Sigma^1\rangle$ is a differential ideal then $\tilde\Sigma^1=\Sigma^1$. For $p>1$ since $\tilde \Sigma^{p-1}$ does not, by assumption,  generate a differential ideal, there are two cases for $\Sigma^{p-1}$. The first case, if $\Sigma^{p-1}$ generates a differential ideal then $\Sigma^{p-1}=\Sigma^p$ and so $\Sigma^{p-1}=\tilde \Sigma^p=\Sigma^p$ by Lemma \ref{relation-Sigma-Sigmatilde}. The second case, if $\Sigma^{p-1}$ does not generate a differential ideal, then $\langle\Sigma^p\rangle$ is a differential ideal because $\langle\tilde \Sigma^p\rangle$ is a differential ideal and $\tilde \Sigma^p=\Sigma^p=\tilde \Sigma^{p+1}$.
\end{proof}

The following proposition indicates the sufficient condition for the nonexistence of regular solutions. This can be used to exclude the cases where there are no regular solutions.
\begin{propn}\label{degenerate-condition}
Suppose $\bold\Phi$ is diagonalisable with distinct eigenvalues. If there is some $\omega^c$ missing in the final submodule $\Sigma^f$, that is $\omega(X_c^V,X_c^H)=0$ for all $\omega \in \Sigma^f$, then there is no regular solution to the inverse problem.
\end{propn}

\begin{proof}
Let $\omega \in\Sigma^f$. We have $\Gamma \hook \omega=0$.
If $\omega^c$ is missing in $\Sigma^f$, then $X_c^{V} \hook \omega=0$ and $X_c^{H} \hook \omega=0$. It then follows that $\omega$ has kernel of dimension greater than one.
\end{proof}

\begin{thm}\label{DI-first-step-cond}
Let $\bold \Phi$ be diagonalisable with distinct eigenvalues. Suppose there are $q$ non-integrable eigen co-distributions. If the sequence $\langle \tilde \Sigma^1 \rangle, ..., \langle \tilde \Sigma^q \rangle$ does not contain a differential ideal then there is no non-degenerate solution.
\end{thm}


\begin{proof}
Suppose that the eigen co-distributions are ordered so that the first $q$ are non-integrable.
 Firstly, if $\langle \tilde \Sigma^q \rangle$ is not a differential ideal, then no earlier $\langle \tilde \Sigma^p \rangle$ can be a differential ideal. Now each of the $n-q$ integrable $\omega^b:=\phi^{bV}\wedge \phi^{bH}$ has remained in $\tilde \Sigma^q$ since $d\omega^b=\xi^b_b\wedge\omega^b$. However $\langle\tilde\Sigma^q\rangle$ is not a differential ideal so that $dim(\tilde \Sigma^q)> n-q$.
Now $dim(\tilde \Sigma^{p+1})<dim(\tilde \Sigma^p)$ for $p<q+1$ and so $dim(\tilde \Sigma^q)\leq n-(q-1)$. Thus $dim(\tilde\Sigma^q) = n-q+1$. But $\langle\tilde\Sigma^q \rangle$ is not a differential ideal by assumption and hence $dim(\tilde \Sigma^{q+1})=n-q$ and so $\omega^1, ...,\omega^q$ are missing and no solution exists.
\end{proof}
We remark that it is not the case that the number of non-integrable co-distributions always matches the terminating differential ideal step. Example \ref{counterxmpl} in section \ref{examples} demonstrates this for $n=3$.
\begin{cor}
Let $\bold \Phi$ be diagonalisable with distinct eigenvalues. If the final submodule is one dimensional and $\tilde\Sigma^f=Sp\{r_a\omega^a  , a=1,\dots,n, r_a\neq 0\}$ , then no eigen co-distributions of $\bold \Phi$ are integrable.
\end{cor}

For the sake of completeness we reproduce the following theorem about the limiting cases.

\begin{thm}\label {PK thm} (see \cite{PK07})
Suppose that the final differential ideal is generated by a one dimensional submodule $\Sigma^f = Sp\{\tilde\omega\},$ for non-degenerate $\tilde\omega.$ That is, there exists $\mu$ such that
$$d\tilde\omega = \mu\wedge\tilde\omega, \quad
\wedge^n\tilde\omega \neq 0.$$
Then
$\langle\tilde\omega\rangle$ contains a closed, non-degenerate two-form if and
only if $d\mu=0$.
\end{thm}

We characterise non-integrable eigen co-distributions in the next theorem.

\begin{thm}\label{integrable-direction-cond-diff}
Let $\Phi$ be diagonalisable with distinct eigenvalues and suppose $\tilde \Sigma^1=Sp\{\phi^{bV}\wedge \phi^{bH}, b=1,...,n\}$ generates a differential ideal. Suppose further that $Sp\{\phi^{aV},\phi^{aH}\}$ is a non-integrable eigen co-distribution of $\bold \Phi$ for some $a$. Then \\
1.\ there exists at least one non-zero $\tau^{aV}_{bb}$ for some $b \neq a$;\\
2.\ let $\tau^{aV}_{bb}, \tau^{aH}_{bb} \neq 0$ for some $b \neq a$ then $\bar \alpha_a=\phi^{aV}+\bar B_a\phi^{aH}$ (no sum) is an integrable direction in $Sp\{\phi^{aV},\phi^{aH}\}$ if and only if
\begin{align}
\label{integrable-direction-cond-diff-1}&\bar B_a=\frac{\tau^{aH}_{bb}}{\tau^{aV}_{bb}} \quad \text{for all such b},\\
\label{integrable-direction-cond-diff-2}\text{and }&X_b^V(\bar B_a)=\bar B_a\tau^{aV}_{ab}-\tau^{aH}_{ab}, \quad a \neq b.
\end{align}
3.\ Let $\tau^{aH}_{bb}=0$ for all $b \neq a$ then $\phi^{aV}$ is an integrable direction if and only if $\tau^{aH}_{ab}=0$ for all $b \neq a$.
\end{thm}

\begin{proof}
\begin{itemize}
\item[1.]\ Since the co-distribution $Sp\{\phi^{aV},\phi^{aH}\}$ is non-integrable and $\tilde \Sigma^1$ is a differential ideal, there is at least one $\tau^{aV}_{bb}\neq 0$ by Proposition \ref{comp-in-cond}.\\

\item[2.]\ Suppose that $\bar \alpha_a=\phi^{aV}+\bar B_a\phi^{aH}$ is an integrable direction in $Sp\{\phi^{aV},\phi^{aH}\}$, that is \eqref{integrable-direction-cond-1}-\eqref{integrable-direction-cond-7} hold for $B_a= \bar B_a$. Then from \eqref{integrable-direction-cond-2} with $c=b$ we get $$\bar B_a=\frac{\tau^{aH}_{bb}}{\tau^{aV}_{bb}},$$ and \eqref{integrable-direction-cond-diff-2} is exactly \eqref{integrable-direction-cond-6}.\\
Conversely, we show that if $\bar B_a=\frac{\tau^{aH}_{bb}}{\tau^{aV}_{bb}}$, and $X_b^V(\bar B_a)=\bar B_a\tau^{aV}_{ab}-\tau^{aH}_{ab}$, then \eqref{integrable-direction-cond-1}-\eqref{integrable-direction-cond-7} hold for $B_a=\bar B_a$ as follows.\\
The conditions \eqref{integrable-direction-cond-1},\eqref{integrable-direction-cond-2},\eqref{integrable-direction-cond-3} and \eqref{integrable-direction-cond-4} hold for $B_a=\bar B_a$ from Corollary \ref{firststep-condition}.\\
To establish \eqref{integrable-direction-cond-5} we note that
$$\tau^{aH}_{bc}=-\Gamma(\tau^{aV}_{bc}) \quad \text{for all $b,c \neq a$}$$
from Jacobi identity \eqref{bracket-iden-1-a} and $\tau^{a\Gamma}_b=0$ for all $a,b$, and
$$\Gamma(\tau^{aH}_{bb})=\lambda_a\tau^{aV}_{bb}$$
from Jacobi identity \eqref{bracket-iden-1-b} together with the conditions in Corollary \ref{firststep-condition}.\\
We have
\begin{align*}
\Gamma(\bar B_a)&=\Gamma(\frac{\tau^{aH}_{bb}}{\tau^{aV}_{bb}}) \quad \text{for } \tau^{aV}_{bb}\neq 0\\
&=\frac{\Gamma(\tau^{aH}_{bb})- \bar B_a\Gamma(\tau^{aV}_{bb})}{\tau^{aV}_{bb}}\\
&=\lambda_a+(\bar B_a)^2,
\end{align*}
so the condition \eqref{integrable-direction-cond-5} holds.\\
The condition \eqref{integrable-direction-cond-diff-2} is exactly \eqref{integrable-direction-cond-6} and implies the condition \eqref{integrable-direction-cond-7} for $B_a=\bar B_a$, $X_b^H(\bar B_a)=\phi^{aV}(R(X_b^H,X_a^H))+\bar B_a(\bar B_a\tau^{aV}_{ab}-\tau^{aH}_{ab})$ as follows.
\begin{align*}
X_b^H(\bar B_a)&=-[\Gamma,X_b^V](\bar B_a)=X_b^V(\Gamma(\bar B_a))-\Gamma(X_b^V(\bar B_a))\\
          &=X_b^V((\bar B_a)^2+\lambda_a)-\Gamma(\bar B_a\tau^{aV}_{ab}-\tau^{aH}_{ab})\\
          &=2\bar B_a(\bar B_a\tau^{aV}_{ab}-\tau^{aH}_{ab})+X_b^V(\lambda_a)-((\bar B_a)^2+\lambda_a)\tau^{aV}_{ab}+\bar B_a\tau^{aH}_{ab}+\Gamma(\tau^{aH}_{ab})
\end{align*}
By identity \eqref{bracket-iden-1-b} and Corollary \ref{firststep-condition},
$$\Gamma(\tau^{aH}_{ab})=\phi^{aV}(R(X_b^H,X_a^H))+\lambda_a\tau^{aV}_{ab}-X_b^V(\lambda_a)$$
Substituting this into $X_b^H(\bar B_a)$ above and then simplifying we get

$$X_b^H(\bar B_a)=\phi^{aV}(R(X_b^H,X_a^H))+\bar B_a(\bar B_a\tau^{aV}_{ab}-\tau^{aH}_{ab})$$

as required.\\

\item[3.]\ If $\phi^{aV}$ is integrable then $d\phi^{aV}=\mu^a\wedge\phi^{aV}$. By looking at \eqref{dphiaV} along with the assumption of the differential ideal, $\langle\tilde\Sigma^1\rangle,$ we have $\tau^{aH}_{ab}=0$ for $b \neq a$.\\
Conversely, using identity \eqref{bracket-iden-1-b} and Corollary \ref{firststep-condition} to prove that $\phi^{aV}(R(X_b^H,X_a^H))=0$ and so $\phi^{aV}(R(X_a^H,X_b^H))=0$ because $R$ is skew, and all other terms apart from the form $\mu \wedge \phi^{aV}$ in the right hand side of \eqref{dphiaV} go because of Corollary \ref{firststep-condition} and the assumptions of the theorem. Hence $\tau^{aH}_{ab}=0$ together with the stated assumptions and \eqref{dphiaV} imply that $\phi^{aV}$ is an integrable direction.
\end{itemize}
\end{proof}

We now examine the consequences of more than one non-zero $\tau^{aV}_{bb}.$

\begin{cor}\label{integrable-direction-cond-diff-cor}
Let $\Phi$ be diagonalisable with distinct eigenvalues and suppose $\tilde \Sigma^1=Sp\{\phi^{bV}\wedge \phi^{bH}, b=1,...,n\}$ is a differential ideal. Suppose that for some $a$, $Sp\{\phi^{aV},\phi^{aH}\}$ is a non-integrable co-distribution of $\bold \Phi$. Suppose further that there exist at least two $\tau^{aV}_{b_ib_i}\neq 0$ for $b_i\neq a$. Then $\bar\alpha_a=\phi^{aV}+\bar B_a\phi^{aH}$ is an integrable direction in $Sp\{\phi^{aV},\phi^{aH}\}$ if and only if $$\bar B_a=\frac{\tau^{aH}_{b_ib_i}}{\tau^{aV}_{b_ib_i}} \quad \text{for each such $b_i \neq a$ }.$$
\end{cor}

\begin{proof}

If $\bar \alpha_a=\phi^{aV}+\bar B_a\phi^{aH}$ is an integrable direction, then $$\bar B_a=\frac{\tau^{aH}_{b_ib_i}}{\tau^{aV}_{b_ib_i}} \quad \text{for all such $b_i \neq a$ }$$ by theorem \ref{integrable-direction-cond-diff} (along with \eqref{integrable-direction-cond-diff-2}). Conversely, we will show that, with the given $\bar B_a,$ $X_{b_i}^V(\bar B_a)=\bar B_a\tau^{aV}_{ab_i}-\tau^{aH}_{ab_i}, \quad b_i \neq a$.
We note that with the conditions in the Corollary \ref{firststep-condition} we get
\begin{equation}\label{bracket-iden-3-aa}
X_{b_i}^V(\tau^{aV}_{b_jb_j})=\tau^{aV}_{b_jb_j}(2\tau^{b_jV}_{b_ib_j}-\tau^{b_jV}_{b_jb_i})-\tau^{aV}_{b_jb_j}\tau^{aV}_{b_ia} \quad \text{for distinct $a, b_i,b_j$}
\end{equation}
from the Jacobi identity \eqref{bracket-iden-3-a}, and
\begin{align}\label{bracket-iden-3-bb}
X_{b_i}^V(\tau^{aH}_{b_jb_j})=\tau^{aH}_{b_jb_j}(2\tau^{b_jV}_{b_ib_j}-\tau^{b_jV}_{b_jb_i})-\tau^{aV}_{b_jb_j}\tau^{aH}_{b_ia} - \tau^{aH}_{b_jb_j}(\tau^{aV}_{b_ia}-\tau^{aV}_{ab_i})\quad \text{for distinct $a, b_i,b_j$}
\end{align}
from the Jacobi identity \eqref{bracket-iden-3-b}.
Now we have
\begin{align*}
X_{b_i}^V(\bar B_a)&=X_{b_i}^V\left(\frac{\tau^{aH}_{b_jb_j}}{\tau^{aV}_{b_jb_j}}\right)       =\frac{X_{b_i}^V(\tau^{aH}_{b_jb_j})-\bar B_aX_{b_i}^V(\tau^{aV}_{b_jb_j})}{\tau^{aV}_{b_jb_j}}.
\end{align*}
Substituting $X_{b_i}^V(\tau^{aV}_{b_jb_j})$ from \eqref{bracket-iden-3-aa} and $X_{b_i}^V(\tau^{aH}_{b_jb_j})$ from \eqref{bracket-iden-3-bb} into $X_{b_i}^V(\bar B_a)$ above and then simplifying we get $$X_{b_i}^V(\bar B_a)=\bar B_a\tau^{aV}_{ab_i}-\tau^{aH}_{ab_i}, \quad b_i \neq a$$ as required.

\end{proof}

\section{Douglas's case IIa2 and an extension}\label{solve-case-2a2-n}
In this section we shall explicitly deal with the inverse problem in arbitrary dimension $n$ given by a diagonalisable $\bold \Phi$ with distinct eigenvalues with exactly $n-1$ co-distributions being integrable. Our result here is given in Theorem \ref{case-2a2-n-results} at the end of the section. We note that this case may be considered as the extension of case IIa2 ($[\nabla\Phi, \Phi]=0$) or case III ($[\nabla\Phi, \Phi]\neq 0$) in the sense of Douglas' classification. In $n=2$, $[\nabla\Phi, \Phi]=0$ is equivalent to $\tilde \Sigma^1$ being a differential ideal. In higher dimensions the problem is rather complicated. In the $n=3$ case for instance, $[\nabla\Phi, \Phi]=0$ does not mean $\tilde \Sigma^1$ is a differential ideal anymore; and the cases of one non-integrable co-distribution and two non-integrable co-distributions may both be considered for the extension of Douglas' case IIa2.

 Consider a given system of second-order ordinary differential equations
\[
\ddot x^a = F^a (t, x^b, \dot x^b), \ \ a, b = 1, \dots, n,
\] for which $\bold \Phi$ is diagonalisable with distinct eigenvalues with exactly $n-1$ co-distributions being integrable. As shown in Theorem \ref{DI-first-step-cond} this system can be variational only when $\tilde \Sigma^1$ is a differential ideal. Let's suppose that we have differential ideal at the first step, i.e. $\tilde \Sigma^1=Sp\{\phi^{cV}\wedge\phi^{cH}:c=1,...,n\}$ generates a differential ideal. Without loss of generality, let us also assume that the only non-integrable co-distribution is $Sp\{\phi^{bV}, \phi^{bH}\}$ for some $b \in \{1,2,...,n\}$ and the other $n-1$ co-distributions, $Sp\{\phi^{aV},\phi^{aH}: a \neq b \}$, are integrable. Then we have:
\begin{align*}
d\omega^b&=d(\phi^{bV}\wedge \phi^{bH})\\
        &=\xi^b_b\wedge \omega^b+\xi^b_a\wedge\omega^a, \quad a\neq b \text{ and } \xi^b_a \neq 0 \text{ for some $a\neq b$},\\
d\omega^a&=d(\phi^{aV}\wedge \phi^{aH})=\xi^a_a \wedge \omega^a \quad (\text{no sum}) \quad a \neq b\\
\end{align*}
where with $A^{aV,H}_{cd}:=\tau^{aV,H}_{cd}-2\tau^{aV,H}_{dc}$,
\begin{equation}\label{xi-defn}
\xi^a_a=A^{aV}_{ac}\phi^{cV}+A^{aH}_{ac}\phi^{cH},\ \xi^b_b=A^{bV}_{bc}\phi^{cV}+A^{bH}_{bc}\phi^{cH},\ \xi^b_a=\tau^{bV}_{aa}\phi^{bV}+\tau^{bH}_{aa}\phi^{bH}.
\end{equation}

We are now looking for 2-forms $\omega \in \tilde \Sigma^1$, i.e. $\omega=r_c\omega^c, c=1,...,n$ with $d\omega=0$ and all the $r$'s non-zero for non-degenerate solutions. We have
$$d\omega=\sum_a(dr_a+r_b\xi^b_a+r_a\xi^a_a)\wedge\phi^{aV}\wedge\phi^{aH}+(dr_b+r_b\xi^b_b)\wedge\phi^{bV}\wedge\phi^{bH},\quad a \neq b .$$
Putting $d\omega=0$, we get a system of Pfaffian equations:
\begin{align}
&dr_b+r_b\xi^b_b=-P_b\phi^{bV}-Q_b\phi^{bH}, \quad \text{$b$ is fixed,}  \label{pfaffian2-1}\\
&dr_a+r_b\xi^b_a+r_a\xi^a_a=-P_a\phi^{aV}-Q_a\phi^{aH} \quad a\neq b \ (\text{no sum on $a$}), \label{pfaffian2-2}
\end{align}
where $P_a$'s, $Q_a$'s, $P_b$ and $Q_b$ are arbitrary functions on $E$.\\
Following the EDS procedure, we extend $E$ to a new manifold $N$ with coordinates $(t,x^c,u^c,r_c,P_c,Q_c)$ and now the problem is to find the integrable distributions on $N$ with $\sigma_c=0$ where
\begin{align}
\sigma_b:=&dr_b+r_b\xi^b_b+P_b\phi^{bV}+Q_b\phi^{bH}  \label{EDS-sigma-1}\\
\sigma_a:=&dr_a+r_b\xi^b_a+r_a\xi^a_a+P_a\phi^{aV}+Q_a\phi^{aH} \quad a\neq b \ (\text{no sum on $a$})\label{EDS-sigma-a}
\end{align}
Continuing the EDS process, set $\pi_c^V:=dP_c$ and $\pi_c^H:=dQ_c$, $c=1,...,n$. (The $V$ and $H$ superscripts here do not indicate that the forms are vertical or horizontal.) Using this a co-frame on $N$ is $(dt, \phi^{cV}, \phi^{cH}, \sigma_c,\pi^V_c, \pi^H_c)$ for $c=1,...,n$. So the next step is to calculate $d\sigma_c$ modulo the ideal
$\langle\sigma_c\rangle$.

Taking the exterior derivative of \eqref{EDS-sigma-a} gives
\begin{align}
\notag d \sigma_a &= d r_b \wedge \xi^b_a + r_b d \xi^b_a + dr_a\wedge\xi^a_a+r_ad\xi^a_a+\pi^V_a \wedge \phi^{aV} + P_a d \phi^{aV} + \pi^H_a \wedge \phi^{aH} + Q_a d \phi^{aH} \\
\label{d-sig}  &\equiv (- r_b \xi^b_b - P_b \phi^{bV} - Q_b \phi^{bH}) \wedge \xi^b_a + r_b d \xi^b_a \\
\notag            &\quad+(- r_b \xi^b_a-r_a\xi^a_a - P_a \phi^{aV} - Q_a \phi^{aH}) \wedge \xi^a_a + r_a d \xi^a_a \\
\notag            &\quad + \pi^V_a \wedge \phi^{aV} + P_a d \phi^{aV} + \pi^H_a \wedge \phi^{aH} + Q_a d \phi^{aH}\ (\text{mod } \sigma_c )\text{(no sum on a).}
\end{align}
The next step is to see what terms in $d\sigma_a$ can be absorbed into $\pi^V_a$ and $\pi^H_a$.  Given that in
each $d \sigma_a$, any term that can be written as $\beta \wedge \phi^{aV}$ or $\beta \wedge \phi^{aH}$ can be
absorbed into terms $\pi^V_a \wedge \phi^{aV}$ and $\pi^H_a \wedge \phi^{aH}$ respectively. After this absorption these terms are denoted as $\tilde{\pi}^V_a \wedge \phi^{aV}$  and $\tilde{\pi}^H_a \wedge \phi^{aH}$ and the remainder that can't be absorbed represents the \textit{`torsion'} of the
system.\\
Working through \eqref{d-sig}, it can be seen from \eqref{dxi-aa} and the integrability of eigen co-distributions $Sp\{\phi^{aV},\phi^{aH}, a \neq b\}$ that terms $d\xi^a_a$, $d\phi^{aV}$, and $d\phi^{aH}$ give no torsion and terms that contribute to the torsion are (remember that $b$ is fixed)
\begin{equation}\label{torsion}
T_a:= \left(r_b(\xi^a_a-\xi^b_b)-P_b\phi^{bV}-Q_b\phi^{bH}\right)\wedge \xi^b_a +r_b d\xi^b_a \quad \text{(no sum on a)}.
\end{equation}

By looking at \eqref{torsion}, it can be seen that for those $a \neq b$ where $\xi^b_a=0$ or equivalently $\tau^{bV}_{aa}=0$, the torsion $T_a$ vanishes without any extra conditions. It then follows that
$$d \sigma_a \equiv \tilde{\pi}^V_a \wedge \phi^{aV}+\tilde{\pi}^H_a \wedge \phi^{aH} \ (\text{mod } \sigma)\ (\text{no sum}).$$
However, the eigen co-distribution $D^{\perp}_b=Sp\{\phi^{bV},\phi^{bH}\}$ is non-integrable by assumption, so there exists at least one $\xi^b_a\neq 0, a \neq b$; and each such $a$ corresponds to one $T_a$. So we split the problem into two subcases:
\begin{itemize}
\item[i)] There is only one fixed $a$ such that $\xi^b_a\neq 0$ or equivalently $\tau^{bV}_{aa}\neq 0$.

\item[ii)] There is more than one fixed $a$ such that $\xi^b_a\neq 0$, say there are $a_i \neq b, i=1,2,...$ such that $\tau^{bV}_{a_ia_i}\neq 0$.
\end{itemize}

Considering $a \neq b$ with $\xi^b_a\neq 0$, and remembering that $b$ is fixed, computing $d\xi^b_a$ we get:
\begin{align*}
\notag d\xi^b_a &=d\tau^{bV}_{aa}\wedge\phi^{bV}+\tau^{bV}_{aa}d\phi^{bV}+d\tau^{bH}_{aa}\wedge\phi^{bH}+\tau^{bH}_{aa}d\phi^{bH}\\
&=(\Gamma(\tau^{bV}_{aa})dt+X_c^V(\tau^{bV}_{aa})\phi^{cV}+X_c^H(\tau^{bV}_{aa})\phi^{cH})\wedge\phi^{bV}\\
&\quad+\tau^{bV}_{aa}(-\lambda_bdt\wedge\phi^{bH}+\tau^{bH}_{cb}\phi^{bV}\wedge\phi^{cH}+\tau^{bV}_{cb}\phi^{bV}\wedge\phi^{cV}-\frac{1}{2}\phi^{bV}(R(X_b^H,X_c^H))\phi^{bH}\wedge\phi^{cH})\\
&\quad+(\Gamma(\tau^{bH}_{aa})dt+X_c^V(\tau^{bH}_{aa})\phi^{cV}+X_c^H(\tau^{bH}_{aa})\phi^{cH})\wedge\phi^{bH}\\
&\quad+\tau^{bH}_{aa}(dt\wedge\phi^{bV}+\tau^{bH}_{cb}\phi^{bH}\wedge\phi^{cH}-\tau^{bV}_{bc}\phi^{bV}\wedge\phi^{cH}).
\end{align*}
And so,
\begin{align*}
\notag T_a& \equiv(r_b(X_b^V(\tau^{bH}_{aa})-X_b^H(\tau^{bV}_{aa})+\tau^{bH}_{aa}A^{aV}_{ab}-\tau^{bV}_{aa}A^{aH}_{ab}+\tau^{bV}_{aa}\tau^{bH}_{bb}-\tau^{bH}_{aa}\tau^{bV}_{bb})\\
\notag  &\quad -P_b\tau^{bH}_{aa}+Q_b\tau^{bV}_{aa})\phi^{bV}\wedge\phi^{bH}\\
\notag &\quad +(\Gamma(\tau^{bV}_{aa})+\tau^{bH}_{aa})dt\wedge \phi^{bV}\\
\notag &\quad+ (\Gamma(\tau^{bH}_{aa})-\lambda_b\tau^{bV}_{aa})dt\wedge \phi^{bH}\\
\notag & + (X_c^V(\tau^{bV}_{aa})+\tau^{bV}_{aa}A^{aV}_{ac}+\tau^{bV}_{aa}\tau^{bV}_{cb})\phi^{cV}\wedge\phi^{bV}\\
\notag &+\quad (X_c^H(\tau^{bV}_{aa})+\tau^{bV}_{aa}A^{aH}_{ac}+\tau^{bV}_{aa}(\tau^{bH}_{cb}-\tau^{bH}_{bc})+\tau^{bH}_{aa}\tau^{bV}_{bc})\phi^{cH}\wedge\phi^{bV}\\
\notag &+\quad (X_c^V(\tau^{bH}_{aa})+\tau^{bV}_{aa}A^{aV}_{ac}+\tau^{bH}_{aa}(\tau^{bV}_{cb}-\tau^{bV}_{bc})+\tau^{bV}_{aa}\tau^{bH}_{bc})\phi^{cV}\wedge\phi^{bH}\\
\notag &+\quad (X_c^H(\tau^{bH}_{aa})+\tau^{bH}_{aa}A^{aH}_{ac}+\tau^{bH}_{aa}\tau^{bH}_{cb}+\tau^{bV}_{aa}\phi^{bV}(R(X_b^H,X_c^H)))\phi^{cH}\wedge\phi^{bH}\\
\notag &+\quad (\tau^{bV}_{aa}\tau^{bH}_{cc}-\tau^{bH}_{aa}\tau^{bV}_{cc})\phi^{cV}\wedge\phi^{cH}, c\neq a\neq b \ (\text{sum on $c$})\ (\text{mod $\phi^{aV}, \phi^{aH}$}).
\end{align*}
Using Jacobi identities \eqref{bracket-iden-1-a}, \eqref{bracket-iden-1-b}, \eqref{bracket-iden-3-a}, \eqref{bracket-iden-4-a}, \eqref{bracket-iden-3-b} and \eqref{bracket-iden-4-b} respectively, we have that the coefficients of $dt\wedge \phi^{bV}$, $dt\wedge \phi^{bH}$, $\phi^{cV}\wedge\phi^{bV}$, $\phi^{cH}\wedge\phi^{bV}$, $\phi^{cV}\wedge\phi^{bH}$ and $\phi^{cH}\wedge\phi^{bH}$ in $T_a$ vanish. Therefore the torsion is now
\begin{align}
\notag T_a&\equiv(r_b(X_b^V(\tau^{bH}_{aa})-X_b^H(\tau^{bV}_{aa})+\tau^{bH}_{aa}A^{aV}_{ab}-\tau^{bV}_{aa}A^{aH}_{ab}+\tau^{bV}_{aa}\tau^{bH}_{bb}-\tau^{bH}_{aa}\tau^{bV}_{bb})\\
\label{torsion-a}          &\quad -P_b\tau^{bH}_{aa}+Q_b\tau^{bV}_{aa})\phi^{bV}\wedge\phi^{bH}\\
\notag    &+\quad (\tau^{bV}_{aa}\tau^{bH}_{cc}-\tau^{bH}_{aa}\tau^{bV}_{cc})\phi^{cV}\wedge\phi^{cH}, c \neq a \neq b\  (\text{sum on $c$}) \ (\text{mod } \phi^{aV}, \phi^{aH})
\end{align}
The torsion must be zero for the existence of solutions, and if it is then, for each $a \neq b$ with $\xi^b_a \neq 0$,
$$d \sigma_a \equiv \tilde{\pi}^V_a \wedge \phi^{aV}+\tilde{\pi}^H_a \wedge \phi^{aH} \ (\text{mod } \sigma).$$
We now examine the conditions for torsion to vanish. By looking at \eqref{torsion-a}, we get:\\
\begin{itemize}
\item If there is only one $a \neq b$ such that $\xi^b_a\neq 0$, i.e. case i), then this remaining $T_a$ vanishes if and only if
\begin{equation}\label{Q-b-condition}
Q_b=P_bB_b+r_bC,\quad \text{where } B_b=\frac{\tau^{bH}_{aa}}{\tau^{bV}_{aa}}
\end{equation}
and $$C=\frac{1}{\tau^{bV}_{aa}}(X_b^H(\tau^{bV}_{aa})-X_b^V(\tau^{bH}_{aa})-\tau^{bH}_{aa}A^{aV}_{ab}+\tau^{bV}_{aa}A^{aH}_{ab}-\tau^{bV}_{aa}\tau^{bH}_{bb}+\tau^{bH}_{aa}\tau^{bV}_{bb})$$
\item If there exist at least two $a_i$'s such that $\xi^b_{a_i}\neq 0, a_i\neq b$, i.e. case ii), then the conditions for all torsion to vanish are \eqref{Q-b-condition} and with
\begin{align}
\label{condition-B}&B_b=\frac{\tau^{bH}_{a_ia_i}}{\tau^{bV}_{a_ia_i}},\\
\label{condition-C} &C=C_{a_i}=\frac{1}{\tau^{bV}_{a_ia_i}}(X_b^H(\tau^{bV}_{a_ia_i})-X_b^V(\tau^{bH}_{a_ia_i})-\tau^{bH}_{a_ia_i}A^{a_iV}_{a_ib}+\tau^{bV}_{a_ia_i}A^{a_iH}_{a_ib}-\tau^{bV}_{a_ia_i}\tau^{bH}_{bb}+\tau^{bH}_{a_ia_i}\tau^{bV}_{bb}),
\end{align}
for all such $a_i$.
\end{itemize}
In either case i) or ii) the consequence of the choice for $Q_b$ in condition \eqref{Q-b-condition} affects $\sigma_b$ in \eqref{EDS-sigma-1} and hence $d\sigma_b$. Hence by substituting $Q_b$ into $d\sigma_b$ we get
\begin{align}
\notag d\sigma_b&=dr_b\wedge\xi^b_b+r_bd\xi^b_b+dP_b\wedge\phi^{bV}+P_bd\phi^{bV}+d(P_bB_b+r_bC)\wedge\phi^{bH}+(P_bB_b+r_bC)d\phi^{bH}\\
\label{d-sig-b} &\equiv (-r_b\xi^b_b-P_b\phi^{bV}-(r_bC+P_bB_b)\phi^{bH})\wedge \xi^b_b+r_bd\xi^b_b\\
\notag          &\quad +dP_b\wedge \phi^{bV}+P_bd\phi^{bV}+dP_b\wedge B_b\phi^{bH}+P_bd(B_b\phi^{bH})\\
\notag          &\quad +C(-r_b\xi^b_b-P_b\phi^{bV}-(r_bC+P_bB_b)\phi^{bH})\wedge\phi^{bH}+r_bd(C\phi^{bH})\ (\text{mod } \sigma).
\end{align}
Simplifying \eqref{d-sig-b} we get:
\begin{align}
\label{d-sig-b-b} d\sigma_b&\equiv (\pi^V_b+P_b(\xi^b_b+C\phi^{bH}))\wedge(\phi^{bV}+B_b\phi^{bH})\\
\notag &\quad +r_bd(\xi^b_b+C\phi^{bH})+P_bd(\phi^{bV}+B_b\phi^{bH})
 \ (\text{mod } \sigma )
\end{align}
At this point the problem breaks down into two further subcases:
\begin{itemize}

\item[1.] If $d(\phi^{bV}+B_b\phi^{bH})=\kappa\wedge(\phi^{bV}+B_b\phi^{bH})$, for some 1-form $\kappa$, then the condition for the existence of non-degenerate solutions is that
    \begin{align}\label{d-sig-b-condition}
    d(\xi^b_b+C\phi^{bH})=\beta\wedge(\phi^{bV}+B_b\phi^{bH}), \text{for some 1-form}\ \beta
    \end{align}
    and then $$d\sigma_b\equiv \tilde{\pi}_b^V \wedge (\phi^{bV}+B_b\phi^{bH})\ (\text{mod }\sigma).$$

\item[2.] If $d(\phi^{bV}+B_b\phi^{bH})\neq\kappa\wedge(\phi^{bV}+B_b\phi^{bH})$, then in order to remove torsion we require $r_bd(\xi^b_b+C\phi^{bH})+P_bd(\phi^{bV}+B_b\phi^{bH})\equiv 0 \ (\text{mod }\phi^{bV}+B_b\phi^{bH})$. This results in an equation relating $P_b$ to $r_b$ which would fix $P_b$ as a function of $r_b$, thus we will have lost flexibility in $\pi_b^V=d(P_b)$ to absorb any terms. So in this situation, the problem reduces to finding a solution for $P_b$ in term of $r_b$ to the equation
    \begin{align}
    \label{d-sig-b-condition-equation} (dP_b+P_b(\xi^b_b+C\phi^{bH}))\wedge(\phi^{bV}+B_b\phi^{bH})+r_bd(\xi^b_b
    +C\phi^{bH})+P_bd(\phi^{bV}+B_b\phi^{bH})=0.
    \end{align}
    Thus if there exists a function $P_b$ in term of $r_b$ satisfying equation \eqref{d-sig-b-condition-equation}, then we have
    $$d\sigma_b\equiv 0 \ (\text{mod }\sigma).$$

\end{itemize}

Let assume that we are in subcase $1.$ so that there exists an integrable direction $\alpha_b=\phi^{bV}+B_b\phi^{bH}$ and so assume that $d(\xi^b_b+C\phi^{bH})=\beta\wedge(\phi^{bV}+B_b\phi^{bH})$, since otherwise there would be no non-degenerate solution. Then we move onto the calculation of the freedom in the solution to the inverse problem for this case, we have
\begin{align}
\label{d-sigma-1a}d\sigma_b&\equiv\tilde{\pi}_b^V\wedge(\phi^{bV}+B_b\phi^{bH}),\\
\label{d-sigma-1b}d\sigma_a&\equiv\tilde{\pi}_a^V\wedge\phi^{aV}
+\tilde{\pi}_a^H\wedge\phi^{aH}.
\end{align}
We change the basis $\{\phi^{kV},\phi^{kH}\}$  to the basis $\{\gamma^{kV},\gamma^{kH}\}$ using
\begin{align*}
\gamma^{1V,H}=\phi^{1V,H}+\phi^{2V,H}+...+\phi^{nV,H},\\
\gamma^{cV,H}=\phi^{1V,H}-\phi^{cV,H}, \quad c=2,...,n.
\end{align*}
We then get the optimal tableau:
\begin{center}
\renewcommand{\arraystretch}{1.25}
$\tilde \Pi = $
\begin{tabular}{c|c c c c c c c c c c}
           & $\gamma^{1V}$  & $\gamma^{1H}$  & $\gamma^{2V}$ & $\gamma^{2H}$ & ...    & $\gamma^{bV}$ & $\gamma^{bH}$&...& $\gamma^{nV}$ & $\gamma^{nH}$  \\ \hline
$\sigma_1$ &$\tilde \pi^V_1$&$\tilde\pi^H_1$&      $\tilde \pi^V_1$    &      $\tilde \pi^H_1$    & ...    & $\tilde \pi^V_1$   & $\tilde \pi^H_1$&...& $\tilde \pi^V_1$   & $\tilde \pi^H_1$  \\
$\sigma_2$ &     $\tilde \pi^V_2$     &     $\tilde \pi^H_2$     &  $-\tilde \pi^V_2$  &  $-\tilde\pi^H_2$  & ...    & $0$   &  $0$&...& $0$   & $0$  \\
\vdots     &    \vdots   &  \vdots     &      \vdots &  \vdots     & ...    &  \vdots &  \vdots&...&  \vdots &  \vdots \\
$\sigma_b$ &     $\tilde \pi^V_b$     &     $B_b\tilde \pi^V_b$     &      $0$    &      $0$    & ...    & $-\tilde \pi^V_b$   & $-B_b\tilde \pi^V_b$ &...&  $0$  &  $0$\\
\vdots     &    \vdots   &  \vdots     &      \vdots &  \vdots     & ...    &  \vdots &  \vdots&...&  \vdots &  \vdots \\
$\sigma_n$ &     $\tilde \pi^V_n$     &     $\tilde \pi^H_n$     &      $0$    &      $0$    & ...    & $0$   & $0$ &...&  $-\tilde \pi^V_n$  &  $-\tilde \pi^H_n$
\end{tabular}
\end{center}
This tableau gives Cartan characters: $s_1=n$, $s_2=n-1$, $s_i=0$ for $i\geq 3$.\\
The final step is to check for involution. To do this, we let $t$ be the number of ways that $\tilde \pi^V_a$, $\tilde\pi^H_a$ and $\tilde \pi^V_b$ can be altered such that \eqref{d-sigma-1a} and \eqref{d-sigma-1b} are unchanged. It can be seen that if we write:
\begin{align*}
\bar\pi^{V}_a&=\tilde\pi^{V}_a+f_a^1\phi^{aV}+f_a^2\phi^{aH},\\
\bar\pi^{H}_a&=\tilde\pi^{H}_a+f_a^3\phi^{aH}+f_a^2\phi^{aV},\\
\bar\pi^{V}_b&=\tilde\pi^{V}_b+f_b(\phi^{bV}+B_b\phi^{bH}),
\end{align*}
then \eqref{d-sigma-1a} and \eqref{d-sigma-1b} would be unchanged if we replace $\tilde \pi^{V,H}_a$ by $\bar \pi^{V,H}_a$ and $\tilde \pi^{V}_b$ by $\bar \pi^{V}_b$. Thus for each $a \neq b$ we have three degrees of freedom in adding terms to $\tilde\pi^{V}_a$ and $\tilde\pi^{H}_a$, giving $3(n-1)$ degrees of freedom for all $\tilde \pi^{V,H}_a$. We have only 1 degree of freedom in adding terms to $\pi^V_b$. Therefore in this case, $t=3(n-1)+1=3n-2$, which equal to $s_1+2s_2$ as required for involution. So the solution depends on $n-1$ functions of two variables in this case.\\

Now we consider subcase 2. where there is no integrable direction $\alpha_b$, and assume that there is solution for equation \eqref{d-sig-b-condition-equation}, so we have
\begin{align}
\label{d-sigma-2a}d\sigma_b&\equiv0,\\
\label{d-sigma-2b}d\sigma_a&\equiv\tilde{\pi}_a^V\wedge\phi^{aV}
+\tilde{\pi}_a^H\wedge\phi^{aH}.
\end{align}
The tableau corresponding with this system is
\begin{center}
\renewcommand{\arraystretch}{1.25}
$\tilde \Pi = $
\begin{tabular}{c|c c c c c c c c c c}
           & $\gamma^{1V}$  & $\gamma^{1H}$  & $\gamma^{2V}$ & $\gamma^{2H}$ & ...    & $\gamma^{bV}$ & $\gamma^{bH}$&...& $\gamma^{nV}$ & $\gamma^{nH}$  \\ \hline
$\sigma_1$ &$\tilde \pi^V_1$&$\tilde \pi^H_1$&      $\tilde \pi^V_1$    &      $\tilde \pi^H_1$    & ...    & $\tilde \pi^V_1$   & $\tilde \pi^H_1$&...& $\tilde \pi^V_1$   & $\tilde \pi^H_1$  \\
$\sigma_2$ &     $\tilde \pi^V_2$     &     $\tilde \pi^H_2$     &  $-\tilde \pi^V_2$  &  $-\tilde\pi^H_2$  & ...    & $0$   &  $0$&...& $0$   & $0$  \\
\vdots     &    \vdots   &  \vdots     &      \vdots &  \vdots     & ...    &  \vdots &  \vdots&...&  \vdots &  \vdots \\
$\sigma_b$     &    $0$   & $0$     &      $0$ &  $0$     & ...    &  $0$ &  $0$&...&  $0$ &  $0$ \\
\vdots     &    \vdots   &  \vdots     &      \vdots &  \vdots     & ...    &  \vdots &  \vdots&...&  \vdots &  \vdots \\
$\sigma_n$ &     $\tilde \pi^V_n$     &     $\tilde \pi^H_n$     &      $0$    &      $0$    & ...    & $0$   & $0$ &...&  $-\tilde \pi^V_n$  &  $-\tilde \pi^H_n$
\end{tabular}
\end{center}
This tableau gives Cartan characters: $s_1=n-1$, $s_2=n-1$, $s_i=0$ for $i\geq3$.\\
The final step is to check for involution. To do this, we let t equal the number of ways that $\tilde \pi^V_a$ and $\tilde\pi^H_a$ can be altered such that \eqref{d-sigma-2a} and \eqref{d-sigma-2b} is unchanged. It can be seen that if we write
\begin{align*}
\bar\pi^{V}_a&=\tilde\pi^{V}_a+f_a^1\phi^{aV}+f_a^2\phi^{aH},\\
\bar\pi^{H}_a&=\tilde\pi^{H}_a+f_a^3\phi^{aH}+f_a^2\phi^{aV}
\end{align*}
then \eqref{d-sigma-2a} and \eqref{d-sigma-2b} would be unchanged if we replace $\tilde \pi^{V,H}_a$ by $\bar \pi^{V,H}_a$. Thus for each $a \neq b$ we have three degrees of freedom in adding terms to $\tilde\pi^{V}_a$ and $\tilde\pi^{H}_a$, giving $3(n-1)$ degrees of freedom for all $\tilde \pi^{V,H}_a$. Therefore in this case, $t=3(n-1)$, which equals $s_1+2s_2$ as required for involution. So the solution depends on $n-1$ functions of two variables in this case.\\
We note here that we follow Anderson and Thompson in reporting only the highest order term in the degree of freedom.

We now summarise the results for the inverse problem of this class of second-order ODE by the following theorem.
\begin{thm}\label{case-2a2-n-results}
In the case where $\bold \Phi$ is diagonalisable with distinct eigenvalues and exactly one non-integrable co-distribution there are three possibilities.
\begin{itemize}
\item If $\langle \tilde \Sigma^1 \rangle$ is not a differential ideal, then there is no non-degenerate solution.
\item Suppose $\langle \tilde \Sigma^1 \rangle$ is a differential ideal, and there is an integrable direction in the non-integrable co-distribution. If \eqref{d-sig-b-condition} holds, then the solution depends on $n-1$ arbitrary function of $2$ variables. If \eqref{d-sig-b-condition} does not hold, then there is no solution.
\item Suppose $\langle\tilde \Sigma^1 \rangle$ is a differential ideal, and there is no integrable direction in the non-integrable co-distribution. If \eqref{d-sig-b-condition-equation} admits a solution, then the solution depends on $n-1$ arbitrary function of $2$ variables. If \eqref{d-sig-b-condition-equation} does not admit a solution, then there is no solution.
\end{itemize}
\end{thm}
%
%

\section{Examples \label{examples}}
In this section, we shall provide a number of examples to illustrate our results in previous section. These examples cover most of the subcases of one non-integrable eigen co-distribution case. Besides, we also give two examples (example \ref{counterxmpl} and \ref{counterxmpl-2}) of the case $n=3$ with two non-integrable co-distributions. We note that while example \ref{counterxmpl} may be considered as an extension of Douglas's case IIa2, example \ref{counterxmpl-2} may be of Douglas' case IIIa.

\begin{xmpl} This is an example of non-existence for $n=2$ because there is not a differential ideal at step 1, see theorem \ref{DI-first-step-cond}. This is in Douglas' case IIIb. Consider the system
\begin{equation}
\ddot{x}=\dot{y}, \ \ddot{y} =y
\end{equation}
on an appropriate domain. This example was considered by Prince \cite{P00} who showed by direct calculation that no non-degenerate solution exists. $\bold \Phi$ is given by
\[
\bold \Phi =\left( \begin{array}{cc} 0 & 0 \\
0 & -1
\end{array}\right).
\]
The eigenvalues and corresponding eigenvectors are:
\begin{eqnarray*}
0 &\quad\text{and}\quad& X_1:=(a,0), \\
-1 &\quad\text{and}\quad& X_2:=(0,b)
\end{eqnarray*}
for some parameters $a,b$.
We calculate $\hnabla_\Gamma X_{1}^V$ and $\hnabla_\Gamma X_{2}^V$ to find $\tau^{i\Gamma}_j$, $i,j \in \{1,2\}$. In this case, by choosing $a=1$ we get $\hnabla_\Gamma X_{1}^V=0$, i.e. $\tau^{1\Gamma}_1=0=\tau^{2\Gamma}_1$, but there is no $b\neq 0$ such that $\hnabla_\Gamma X_{2}^V=0$, in particular we have $\tau^{1\Gamma}_2=-\frac{b}{2}$ and $\tau^{2\Gamma}_2=\frac{\Gamma(b)}{b}$. This immediately implies $\Sigma^1$ is not a differential ideal as $\tau^{1\Gamma}_2\neq0$ by corollary \ref{firststep-condition}. The functions $\tau^{kV}_{ij}$ and $\tau^{kH}_{ij}$ are also easy to compute
and apart from other $\tau^{kV,H}_{ij}$ we have
\[
\tau^{1V}_{22}=0, \quad \tau^{2V}_{11}=0.\]

So this is the case where $\bold \Phi$ is diagonalisable with distinct eigenvalues and the eigen co-distribution $Sp\{\phi^{2V}, \phi^{2H}\}$ is integrable and $Sp\{\phi^{1V}, \phi^{1H}\}$  is non-integrable since $\tau^{1\Gamma}_2\neq0$. Therefore there is no non-degenerate solution of the inverse problem since $\langle\tilde \Sigma^1\rangle$ is not a differential ideal by the theorem \ref{DI-first-step-cond}.\\

\end{xmpl}

\begin{xmpl}
This is another non-existence example, this time for $n=4$. The example is in subcase i) but fails the condition of the further subcase 1. Consider the system
\begin{equation}
\ddot{x}=x, \
\ddot{y} = 0, \
\ddot{z}=\dot{y}/\dot{z},\
\ddot{w}=\dot{w}
\end{equation}
on an appropriate domain. We find
\[
\bold \Phi =\left( \begin{array}{cccc} -1 & 0 & 0 &0\\
0 & 0 & 0&0 \\
0&-\frac{\dot{y}}{4\dot{z}^3}&\frac{3\dot{y}^2}{4\dot{z}^4}&0\\
0&0&0&-\frac{1}{4} \end{array}\right).
\]
In this case it is possible to choose the eigenvectors $X_a$ such that $\hnabla_\Gamma X_a^V=0$ and so the eigenvalues and corresponding scaled eigenvectors are
\begin{eqnarray*}
-1 &\quad\text{and}\quad& X_1=(1,0,0,0), \\
0 &\quad\text{and}\quad& X_2=(0,3\dot{y},\dot{z},0), \\
\frac{3\dot{y}^2}{4\dot{z}^4} &\quad\text{and}\quad& X_3=(0,0,\frac{1}{\sqrt{\dot{z}}},0),\\
-\frac{1}{4}&\quad\text{and}\quad& X_4=(0,0,0,\sqrt{\dot{w}}).
\end{eqnarray*}
The non-zero $\tau^{aV}_{bc}$ and $\tau^{aH}_{bc}$ are
\[
\tau^{2V}_{22}=3, \quad \tau^{3V}_{22} =-2\dot{z}\sqrt{\dot{z}}, \quad \tau^{3H}_{22} =\frac{3\dot{y}}{\sqrt{\dot{z}}}, \quad \tau^{3V}_{23} =-\frac{1}{2}, \]
 \[\tau^{3V}_{32} =1, \quad \tau^{3V}_{33} =-\frac{1}{2\dot{z}\sqrt{\dot{z}}}, \quad \tau^{3H}_{33} =-\frac{3\dot{y}}{4\dot{z}\sqrt{\dot{z}}}, \quad \tau^{4V}_{44}=\frac{1}{2\sqrt{\dot{w}}},\quad
\tau^{4H}_{44}=\frac{1}{4\sqrt{\dot{w}}}.\]

These results show we are in the case where $\bold \Phi$ is diagonalisable with distinct eigenvalues and eigen co-distributions $Sp\{\phi^{1V}, \phi^{1H}\}$, $Sp\{\phi^{2V}, \phi^{2H}\}$ and $Sp\{\phi^{4V}, \phi^{4H}\}$ are integrable and $Sp\{\phi^{3V}, \phi^{3H}\}$ is non-integrable and $\langle\tilde \Sigma^1\rangle$ is a differential ideal. In particular, we are in the subcase i) where there is only $\xi^3_2\neq 0$ corresponding to $\tau^{3V}_{22} \neq 0$. We also find that the condition $$X_2^V(B_3)=B_3\tau^{3V}_{32}-\tau^{3H}_{32},$$
where $X_2^V=3\dot{y}\frac{\partial}{\partial \dot{y}}+\dot{z}\frac{\partial}{\partial \dot{z}}$ and $B_3=\frac{\tau^{3H}_{22}}{\tau^{3V}_{22}}=-\frac{3\dot{y}}{2\dot{z}^2}$ is satisfied for an integrable direction $\alpha_3=\phi^{3V}+B_3\phi^{3H}$ in the non-integrable co-distribution. The existence of a  non-degenerate solution depends on whether or not,
$$d(\xi^3_3+C\phi^{3H})=\kappa\wedge \alpha_3.$$
Using the formula for $C$ from \eqref{condition-C}, we get
\begin{align*}
C=C_2&=\frac{1}{\tau^{3V}_{22}}(X_3^H(\tau^{3V}_{22})-X_3^V(\tau^{3H}_{22})-\tau^{3H}_{22}A^{2V}_{23}+\tau^{3V}_{22}A^{2H}_{23}-\tau^{3V}_{22}\tau^{3H}_{33}+\tau^{3H}_{22}\tau^{3V}_{33})\\
 &=\frac{3\dot{y}}{4\dot{z}^3\sqrt{\dot{z}}}(\dot{z}^2-1)
\end{align*} and furthermore we have $d\xi^3_3=d(2\phi^{2V})=0$.
Therefore there is no regular solution since $d(C\phi^{3H})\neq \kappa\wedge \alpha_3$.
\end{xmpl}

\begin{xmpl}
This example is from Aldridge et al \cite{Al06} which was modified from one of Douglas' examples in \cite{D41}. However, in their calculation only the differential ideal step was applied to find the differential ideal. Then they had to use the differential conditions of the original Helmholtz conditions to determine the existence of the solution. By following the scheme given at the end of the previous section we quickly conclude that the solution of this system depends on two functions of 2 variables each. This example is the candidate for our case i) where only one $\tau^{bV}_{aa}\neq 0$ for $a\neq b$.\\
Consider the system
\begin{equation}
\ddot{x}=-x, \
\ddot{y} = y^{-1}(1+\dot{y}^2+\dot{z}^2), \
\ddot{z}=0
\end{equation}
on an appropriate domain. Denoting the derivatives by $u,v,w$, we find
\[
\bold \Phi = \frac{1}{y^2}\left( \begin{array}{ccc} y^2 & 0 & 0 \\
0 & 2(1+w^2) & -2vw \\
0&0&0 \end{array}\right).
\]
Eigenvalues and corresponding eigenvectors $X_a$ chosen so that $\hnabla_\Gamma X_a^V=0$ are
\begin{eqnarray*}
\lambda_1=0 &\quad\text{and}\quad& X_1=(0,vw,1+w^2), \\
\lambda_2=1 &\quad\text{and}\quad& X_2=(1,0,0), \\
\lambda_3=2y^{-2}(1+w^2) &\quad\text{and}\quad& X_3=(0,y,0).
\end{eqnarray*}
The only non-zero functions $\tau^{aV}_{bc}$ and $\tau^{aH}_{bc}$ are
\[
\tau^{3V}_{11}=\frac{v}{y}, \quad \tau^{1V}_{11} =2w, \quad
\tau^{3V}_{31}=w, \quad \tau^{3H}_{11}=-\frac{1+w^2}{y^2}.\]
These results show that our system is in the case that $\bold \Phi$ is diagonalisable with distinct eigenvalues and the eigen co-distributions $Sp\{\phi^{1V}, \phi^{1H}\}$ and $Sp\{\phi^{2V}, \phi^{2H}\}$ are integrable and the third one is not and $\langle\tilde \Sigma^1\rangle$ is a differential ideal. In particular, this is in the subcase i) where there is only one $\tau^{3V}_{11} \neq 0$. Checking the condition for an integrable direction $\alpha_3=\phi^{3V}+B_3\phi^{3H}$ we get $$X_1\left(\frac{\tau^{3H}_{11}}{\tau^{3V}_{11}}\right)=\frac{\tau^{3H}_{11}}{\tau^{3V}_{11}}\tau^{3V}_{31}-\tau^{3H}_{31},$$
where $X_1^V=vw\frac{\partial}{\partial v}+(1+w^2)\frac{\partial}{\partial w}$, holds. Now whether or not the system has solutions depends on the remaining condition:
$$d(\xi^3_3+C\phi^{3H})=\kappa\wedge \alpha_3.$$
Using the formula for $C$ from \eqref{condition-C}, we get $C=0$ and furthermore we have $d\xi^3_3=0$.
Therefore the solution of the system depends on 2 functions of 2 variables.
\end{xmpl}

\begin{xmpl}
This example is in case ii) and the multiplier depends on two arbitrary functions each of two variables and one function of one variable.\\
Consider the system
\begin{equation}
\ddot{x}=z, \
\ddot{y} = x\dot{x}+z\dot{z}, \
\ddot{z}=x
\end{equation}
on an appropriate domain. Again denoting the derivatives by $u,v,w$, we find
\[
\bold \Phi =\left( \begin{array}{ccc} 0 & 0 & -1 \\
-\frac{u}{2} & 0 & -\frac{w}{2} \\
-1&0&0 \end{array}\right).
\]
In this case it is possible to choose the eigenvectors $X_a$ such that $\hnabla_\Gamma X_a^V=0$ and so the eigenvalues and chosen corresponding eigenvectors are:
\begin{eqnarray*}
0 &\quad\text{and}\quad& (0,1,0), \\
1 &\quad\text{and}\quad& (-2,u-w,2), \\
-1 &\quad\text{and}\quad& (2,u+w,2).
\end{eqnarray*}
The $\tau^{aV}_{bc}$ and $\tau^{aH}_{bc}$ are zero except for
\[
\tau^{1V}_{22}=-4, \quad \tau^{1V}_{33} =4.\]

These results show that our system is in the case that $\bold \Phi$ is diagonalisable with distinct eigenvalues and eigen co-distribution $\{\phi^{1V}, \phi^{1H}\}$ is non-integrable and $Sp\{\phi^{2V}, \phi^{2H}\}$ and $Sp\{\phi^{3V}, \phi^{3H}\}$ are integrable and $\langle\tilde \Sigma^1\rangle$ is a differential ideal. We are in subcase ii). In addition, the ratio condition \eqref{condition-B},
$$\frac{\tau^{1H}_{22}}{\tau^{1V}_{22}}=0=\frac{\tau^{1H}_{33}}{\tau^{1V}_{33}}$$
and the condition \eqref{condition-C}, $C_2=0=C_3$ are satisfied. We can conclude that there is an integrable direction inside the non-integrable co-distribution, which is just $\phi^{1V}$. Furthermore we have $\xi^1_1=0$ and so $d(\xi^1_1+C\phi^{1H})=0$. Therefore the conclusion is the solution of the system depends upon $2$ arbitrary functions of $2$ variables (plus one function of one variable).

In order to identify the structure of the  $r$'s, we return to equation \eqref{pfaffian2-1} and \eqref{pfaffian2-2}. Since in this case we know that $Sp\{\phi^{1V}, \phi^{1H}\}$ is non-integrable and the other two are integrable and $\phi^{1V}$ is integrable direction, we have
\begin{align*}
&dr_1+r_1\xi^1_1=-P_1\phi^{1V}-Q_1\phi^{1H},\\
&dr_2+r_1\xi^1_2+r_2\xi^2_2=-P_2\phi^{2V}-Q_2\phi^{2H},\\
&dr_3+r_1\xi^1_3+r_3\xi^3_3=-P_3\phi^{3V}-Q_3\phi^{3H},
\end{align*}
where $P_1, Q_1, P_2, Q_2, P_3$ and $Q_3$ are arbitrary functions.
As we know $Q_1=P_1B_1+r_1C=0$, and computing $\xi^i_j$'s using \eqref{xi-defn} we get $\xi^1_1=\xi^2_2=\xi^3_3=0$, $\xi^1_2=-4\phi^{1V}$ and $\xi^1_3=4\phi^{1V}$. So we have
\begin{subequations}
\begin{align}
&dr_1=-P_1\phi^{1V}, \label{pfaffian-ex-1-1}\\
&dr_2-4r_1\phi^{1V}=-P_2\phi^{2V}-Q_2\phi^{2H}, \label{pfaffian-ex-1-2}\\
&dr_3+4r_1\phi^{1V}=-P_3\phi^{3V}-Q_3\phi^{3H}. \label{pfaffian-ex-1-3}
\end{align}
\end{subequations}

Equation \eqref{pfaffian-ex-1-1} implies $r_1=G_1(\zeta)$ where $G_1$ is an arbitrary function of $\zeta$ with $d\zeta \in Sp\{\phi^{1V}\}$. Since $\phi^{1V}=\frac{1}{2}(-d(uw)+2dv-xdx-zdz)$, putting $d\zeta=-d(uw)+2dv-xdx-zdz$ we have $\zeta=2v-uw-\frac{x^2}{2}-\frac{z^2}{2}$.
Now substituting $r_1$ into equation \eqref{pfaffian-ex-1-2} we get
\begin{equation}\label{pfaffian-ex-1-4}
dr_2=2G_1(\zeta)d\zeta-P_2\phi^{2V}-Q_2\phi^{2H}.
\end{equation}
Since $Sp\{\phi^{2V},\phi^{2H}\}$ is integrable, we know that there exist functions $u_2^1$ and $u^2_2$ such that $Sp\{du^1_2, du_2^2\} = Sp\{\phi^{2V},\phi^{2H}\}$. As $\phi^{2V}=\frac{1}{4}\left(d(w-u)+(z-x)dt\right)$ and $\phi^{2H}=\frac{1}{4}\left(d(z-x)+(u-w)dt\right)$, putting $$du_2^i=f_2^i(d(w-u)+(z-x)dt)+g_2^i(d(z-x)+(u-w)dt),$$ and for $f_2^1=2(w-u)$, $g_2^1=2(z-x)$, $f_2^2=-\frac{z-x}{(w-u)^2+(z-x)^2}$ and $g_2^2=\frac{w-u}{(w-u)^2+(z-x)^2}$ we get
\begin{align*}
&u_2^1=(w-u)^2+(z-x)^2\\
&u_2^2=\tan^{-1}\left(\frac{z-x}{w-u}\right)-t
\end{align*}
Now equation \eqref{pfaffian-ex-1-4} gives
$$r_2=2G(\zeta)+\tilde{r}_2(u_2^1,u_2^2),$$ where $G^\prime=G_1$ and $\tilde{r}_2$ is an arbitrary function of $u_2^1$ and $u_2^2$.\\

Similarly, $Sp\{\phi^{3V},\phi^{3H}\}$ is integrable with
$\phi^{3V}=\frac{1}{4}\left(d(u+w)-(x+z)dt\right)$\\ and $\phi^{3H}=\frac{1}{4}\left(d(x+z)-(u+w)dt\right).$ We find $u_3^1=(u+w)^2-(x+z)^2$ and $u_3^2=\tanh^{-1}\left(\frac{x+z}{u+w}\right)-t$ with $Sp\{du_3^1, du_3^2\}= Sp\{\phi^{3V}, \phi^{3H}\}$. This and equation \eqref{pfaffian-ex-1-3} gives
$$r_3=-2G(\zeta)+\tilde{r}_3(u_3^1,u_3^2).$$
To view the multiplier $g_{ab}$, we translate the $r_{ab}$ into $g_{ab}$ using \eqref{conver-r-to-g}. We get,
\begin{align*}
g_{11}&=\frac{w^2}{4}r_1+\frac{1}{16}(r_2+r_3),\\
g_{22}&=r_1,\\
g_{33}&=\frac{u^2}{4}r_1+\frac{1}{16}(r_2+r_3),\\
g_{12}=g_{21}&=-\frac{w}{2}r_1,\\
g_{13}=g_{31}&=\frac{uw}{4}r_1-\frac{1}{16}(r_2+r_3),\\
g_{23}=g_{32}&=-\frac{u}{2}r_1.\\
\end{align*}
And thus,
\begin{align*}
g_{11}&=\frac{w^2}{4}G_1(\zeta)+\frac{1}{16}(\tilde{r}_2(u_2^1,u_2^2)+\tilde{r}_3(u_3^1,u_3^2)),\\
g_{22}&=G_1(\zeta),\\
g_{33}&=\frac{u^2}{4}G_1(\zeta)+\frac{1}{16}(\tilde{r}_2(u_2^1,u_2^2)+\tilde{r}_3(u_3^1,u_3^2)),\\
g_{12}=g_{21}&=-\frac{w}{2}G_1(\zeta),\\
g_{13}=g_{31}&=\frac{uw}{4}G_1(\zeta)-\frac{1}{16}(\tilde{r}_2(u_2^1,u_2^2)+\tilde{r}_3(u_3^1,u_3^2)),\\
g_{23}=g_{32}&=-\frac{u}{2}G_1(\zeta).
\end{align*}
In summary, the most general Cartan two-form for this example is

$$d\theta_L=\frac{1}{2}G'(\zeta)d\zeta\wedge\phi^{1H} +\frac{1}{32}\left[(2G(\zeta)+\tilde r_2(u^1_2,u^2_2))du^1_2\wedge du_2^2
+ (-2G(\zeta)+\tilde r_3(u^1_3,u^2_3))du^1_3\wedge du_2^3\right]$$
While this form is beguiling it is not the generic solution for this class.

\end{xmpl}
We now introduce two $n=3$ examples with two non-integrable eigen co-distributions. These correspond to two cases of differential ideal step in which the system of ordinary differential equations may possibly have non-degenerate solutions. In the first example, the differential ideal step finishes at step 1, i.e. $\tilde \Sigma^1$ generates a differential ideal which demonstrates that the number of non-integrable co-distributions is not equal to the number of steps in the differential process as discussed following Theorem \ref{DI-first-step-cond}. The second example is for the case that $\langle\tilde\Sigma^2\rangle$ is a differential ideal. A full analysis for these cases will be in our forthcoming paper in which we present a classification for the $n=3$ case.
\begin{xmpl}\label{counterxmpl}
Consider the system
\begin{equation}
\ddot{x}=x\dot{z}, \
\ddot{y} = x, \
\ddot{z}=x
\end{equation}
on an appropriate domain. Again denoting the derivatives by $u,v,w$, we find
\[
\bold \Phi = \left( \begin{array}{ccc} -w & 0 & \frac{u}{2} \\
-1 & 0 & 0 \\
-1&0&0 \end{array}\right).
\]
Eigenvalues and corresponding eigenvectors $X_a$ chosen so that $\hnabla_\Gamma X_a^V=0$ are
\begin{eqnarray*}
\lambda_1=\sqrt{-2u+w^2}-w &\quad\text{and}\quad& X_1=(-\sqrt{-2u+w^2}+w,2,2), \\
\lambda_2=-\sqrt{-2u+w^2}-w &\quad\text{and}\quad& X_2=(\sqrt{-2u+w^2}+w,2,2),  \\
\lambda_3=0 &\quad\text{and}\quad& X_3=(0,1,0).
\end{eqnarray*}
The non-zero functions $\tau^{aV}_{bc}$ and $\tau^{aH}_{bc}$ are

\[
\tau^{1V}_{11}=-\tau^{2V}_{11}=\frac{\sqrt{-2u+w^2}-w}{2(2u-w^2)}, \quad \tau^{1H}_{11} =-\tau^{2H}_{11}=\frac{x}{2(2u-w^2)},\]
\[\tau^{1V}_{12}=-\tau^{2V}_{12}=\frac{3\sqrt{-2u+w^2}+w}{2(2u-w^2)}, \quad \tau^{1H}_{12}=-\tau^{2H}_{12}=\frac{-x}{2(2u-w^2)},\]
\[\tau^{1V}_{21}=-\tau^{2V}_{21}=\frac{3\sqrt{-2u+w^2}-w}{2(2u-w^2)}, \quad \tau^{1H}_{21}=-\tau^{2H}_{21}=\frac{x}{2(2u-w^2)},\]
\[\tau^{1V}_{22}=-\tau^{2V}_{22}=\frac{\sqrt{-2u+w^2}+w}{2(2u-w^2)}, \quad \tau^{1H}_{22}=-\tau^{2H}_{22}=\frac{-x}{2(2u-w^2)}.\]

These results show that our system is in the case that $\bold \Phi$ is diagonalisable with distinct eigenvalues and the two co-distributions $Sp\{\phi^{1V}, \phi^{1H}\}$ and $Sp\{\phi^{2V}, \phi^{2H}\}$ are non-integrable but the differential ideal step finishes at step $1$, that is $\tilde \Sigma^1$ generates a differential ideal. Further examination shows that the solution depends on one arbitrary function of two unknowns as we will show in the future paper.
\end{xmpl}
\begin{xmpl}\label{counterxmpl-2}
Consider the system
\begin{equation}
\ddot{x}=zt, \
\ddot{y} = 0, \
\ddot{z}=x
\end{equation}
on an appropriate domain. Again denoting the derivatives by $u,v,w$, we find
\[
\bold \Phi = \left( \begin{array}{ccc} 0 & 0 & -t \\
0 & 0 & 0 \\
-1&0&0 \end{array}\right).
\]
Eigenvalues and corresponding eigenvectors $X_a$ are
\begin{eqnarray*}
\lambda_1=\sqrt{t} &\quad\text{and}\quad& X_1=(-\sqrt{t},0,1), \\
\lambda_2=-\sqrt{t} &\quad\text{and}\quad& X_2=(\sqrt{t},0,1),  \\
\lambda_3=0 &\quad\text{and}\quad& X_3=(0,1,0).
\end{eqnarray*}
The structure functions $\tau$'s are zero except for

\[
\tau^{1\Gamma}_{1}=\tau^{2\Gamma}_{2}=-\tau^{1\Gamma}_{2}=-\tau^{2\Gamma}_{1}=\frac{1}{4t}.\]

These results show that our system is in the case that $\bold \Phi$ is diagonalisable with distinct eigenvalues and the co-distributions $Sp\{\phi^{1V}, \phi^{1H}\}$ and $Sp\{\phi^{2V}, \phi^{2H}\}$ are non-integrable and the third one is integrable and $\tilde \Sigma^1$ does not generate a differential ideal. Examining further we get, \[d(\omega^1-\omega^2)=\frac{1}{4t} dt \wedge (\omega^1-\omega^2),\] and so $\langle\tilde\Sigma^2\rangle$ is a differential ideal. The Cartan two-form for this example is \[\omega=\frac{G}{\sqrt[4]{t}}(\omega^1-\omega^2)+\tilde{r}_3(u_3^1,u_3^2)\omega^3, \]
where $G$ is a constant, and $\tilde{r}_3$ is an arbitrary function of two variables $u^1_3:=y-vt$ and $u^2_3:=v $.
\end{xmpl}
\section{Conclusion}\label{classification-scheme}

We finish this paper with a new proposal for the classification scheme for the inverse problem in dimension $n$:

\begin{itemize}
\item[A.] $\bold{\Phi}=\lambda I_n$. This is equivalent to $\langle \Sigma^0 \rangle$ being a differential ideal (see proposition \ref{identity thm}).
\item[B.] $\bold{\Phi}$ is diagonalisable with distinct eigenvalues (real or complex). Further subcases will be divided according to the integrability of the lifted two-dimensional  eigen co-distributions of $\bold{\Phi}$ i.e. $q$ co-distributions are non-integrable and $n-q$ are integrable. According to our theorem \ref{DI-first-step-cond}, if up to and including $\langle \tilde \Sigma^q \rangle$ there is no differential ideal, then there is no non-degenerate multiplier. Hence, for each $q$, the subcases to be considered are that a differential ideal is generated at step $1$, step $2$,..., up to step $q$. Subsidiary to this is the consideration of integrable directions within non-integrable eigen co-distributions.
\item[C.] $\bold{\Phi}$ is diagonalisable with repeated eigenvalues. Further subdivision according to integrability will be  similar to case B above.
\item[D.] $\bold{\Phi}$ is not diagonalisable. Further subdivision depends on the integrability of normal forms of $\bold \Phi$.
\end{itemize}

We invite the reader who has persevered thus far to compare this scheme to the geometric translation of Douglas's scheme for $n=2$ to be found in \cite{CSMBP94}. If that scheme was translated into EDS terms the differential ideal conditions would be followed by integrability conditions on the eigen co-distributions. In the light of theorem \ref{DI-first-step-cond} and our examples we maintain that the integrability of the eigen co-distributions must be considered first.

\section*{Acknowledgements}
Thoan Do gratefully acknowledges receipt of a Vietnamese government MOET-VIED scholarship and scholarship support from La Trobe University, and the hospitality of the Australian Mathematical Sciences Institute. Both authors thank Willy Sarlet for useful discussions and his continuing interest. We thanks two anonymous referees for their helpful suggestions.


\begin{thebibliography}{99}

\bibitem{Al03} J.E.\ Aldridge, {\em Aspects of the Inverse Problem
in the Calculus of Variations} Ph.D. Thesis, La Trobe University,
Australia (2003).

\bibitem{Al06} J.E.\ Aldridge, G.E.\ Prince, W. Sarlet and G. Thompson. An EDS approach to the inverse problem in the calculus of variations, {\em  J.\ Math.\ Phys. \/} {\bf 47} (2006) 103508 (22 pages).


\bibitem{AT92}
I.\ Anderson and G.\ Thompson, The inverse problem of the calculus
of variations for ordinary differential equations, {\em  Memoirs
Amer.\ Math.\ Soc.\/} {\bf 98} No. 473 (1992).


\bibitem{Bryant91}
R.L.\ Bryant, S.S.\ Chern, R.B.\ Gardner, H.L.\ Goldschmidt and
P.A.\ Griffiths, {\em  Exterior Differential Systems} (Springer,
Berlin)  (1991).

\bibitem{CPST}
M.\ Crampin, G.E.\ Prince, W.\ Sarlet and G.\ Thompson, The
inverse problem of the calculus of variations: separable systems,
{\it Acta Appl.\ Math.\/} {\bf 57} (1999) 239--254.

\bibitem{Cramp84}
M.\ Crampin, G.E.\ Prince and G.\ Thompson. A geometric version of
the Helmholtz conditions in time dependent Lagrangian dynamics,
 {\em  J.\ Phys.\ A:\ Math.\ Gen. \/} {\bf 17} (1984) 1437--1447.

\bibitem{CSMBP94}
M. Crampin, W. Sarlet, E. Mart\'{\i}nez, G. B. Byrnes and G. E. Prince.
Toward a geometrical understanding of Douglas's solution of the
inverse problem in the calculus of variations.
{\em  Inverse Problems\/} {\bf 10} (1994), 245-260.

\bibitem{DP14}
T. Do and G.E. Prince. An intrinsic and exterior form of the Bianchi identities.
{\em submitted} (2014).

\bibitem{D41}
J.\ Douglas, Solution of the inverse problem of the calculus of
variations, {\em  Trans.\ Am.\ Math.\ Soc.} {\bf 50} (1941)
71--128.

%
%
%
\bibitem{GM2}
J.\ Grifone and Z.\ Muzsnay, {\em Variational principles for
second-order differential equations\/}, World Scientific (2000).

\bibitem{HH01}
H. Helmholtz. \"{U}ber der physikalische Bedeutung des Princips der kleinsten
Wirkung, {\it J. Reine Angew. Math.} {\bf 100}, (1887), 137--166.

\bibitem{Hirsch02}
A. Hirsch.  Die Existenzbedingungen des verallgemeinterten kinetischen Potentialen, {\it Math. Ann.} {\bf
50} (1898), 429--441.

\bibitem{JP01}
M.\ Jerie and G.E.\ Prince, Jacobi fields and linear connections
for arbitrary second order ODE's, {\em J.\ Geom.\ Phys.\/} {\bf
43} (2002) 351--370.

\bibitem{KP08}
O. Krupkov\'{a} and G.E.\ Prince, Second order ordinary differential equation in jet bundles and the inverse problem of the calculus of variation in: {\it Handbook of Global Analysis}, edited by D. Krupka and D. Saunders, Elsevier 2008.

\bibitem{MaPa94}
E. Massa and E. Pagani, Jet bundle geometry, dynamical connections, and the inverse problem of Lagrangian mechanics, {\em Ann. Inst. Henri Poincar\'{e}, Phys. Theor.} {\bf 61} (1994) 17--62.

\bibitem{MFLMR90}
G. Morandi, C. Ferrario, G. Lo Vecchio, G. Marmo and C. Rubano.
The inverse problem in the calculus of variations and the geometry of
the tangent bundle.
{\em  Phys.\ Rep.\/} {\bf 188} (1990), 147--284.

%

\bibitem{P00}
G.E. Prince. The inverse problem in the calculus of variations and its
ramifications in {\it Geometric Approaches to Differential Equations} ed. P.
Vassiliou. Lecture Notes of the Australian Mathematical Society. CUP (2000).

\bibitem{PK07} G.E. Prince and D.M. King. The inverse problem in the calculus of
variations: nonexistence of Lagrangians {\it pp} 131--140 in {\it Differential
Geometric methods on Mechanics and Field theory: Volume in Honour of Willy Sarlet},
edited by F.Cantrijn and B. Langerock, Gent, Academia Press (2007).


\bibitem{S82}
W.\ Sarlet, The Helmholtz conditions revisited. A new approach to
the inverse problem of Lagrangian dynamics, {\em  J.\ Phys.\ A:
Math.\ Gen.\/} {\bf 15} (1982) 1503--1517.

\bibitem{SaCraMa}
W.\ Sarlet, M.\ Crampin and E.\ Mart\'{\i}nez, The integrability
conditions in the inverse problem of the calculus of variations
for second-order ordinary differential equations, {\it Acta Appl.\
Math.\/} {\bf 54} (1998) 233--273.

\bibitem{STP}
W. Sarlet, G. Thompson and G.E. Prince, The inverse problem of
the calculus of variations: the use of geometrical calculus in
Douglas's analysis, {\em Trans.\ Amer.\ Math.\ Soc.\/} {\bf 354}
(2002) 2897--2919.

\bibitem{SVCM95}
W.\ Sarlet, A.\ Vandecasteele, F.\ Cantrijn and E.\ Mart\'{\i}nez,
Derivations of forms along a map: the framework for
time--dependent second--order equations, {\em Diff. Geom.
Applic.\/} {\bf 5} (1995) 171--203.


\bibitem{Son}
N. Ya. Sonin, On the definition of maximal and minimal properties,
{\em Warsaw Univ. Izvestiya}{\bf 1--2}{1886}{1--68 (in Russian)}
%



\end{thebibliography}
\end{document}